\documentclass[10pt,a4paper,reqno]{amsart}

\usepackage{amsfonts,amsthm,amssymb,amsmath,enumerate}
\usepackage{mathrsfs,setspace,multirow, multicol,latexsym,mathtools}
\usepackage[thinlines]{easytable}
\usepackage{tikz-cd}
\usepackage{courier}
\usepackage[mathscr]{euscript}
\usepackage{dirtytalk}
\usepackage{enumitem}
\usepackage{hyperref}
\usepackage{pdfpages}

\usepackage{graphicx,caption,subcaption,float}
\usepackage{wrapfig}

\usepackage[super]{nth}

\newtheorem{theorem}{Theorem}[section]
\newtheorem{lemma}[theorem]{Lemma}
\newtheorem{proposition}[theorem]{Proposition}

\theoremstyle{definition}
\newtheorem{definition}{Definition}[section]

\newtheorem{example}{Example}[section]
\theoremstyle{remark}
\newtheorem{remark}[theorem]{Remark}

\usepackage[export]{adjustbox} 

\newcommand{\R}{\mathbb R}

\newcommand{\LP}{\left(}
\newcommand{\RP}{\right)}
\newcommand{\LS}{\left[}
\newcommand{\RS}{\right]}
\newcommand{\LC}{\left\{}
\newcommand{\RC}{\right\}}
\newcommand{\MV}{\,\middle \vert\,}

\begin{document}

\title[Polynomial parametrization of some interesting knotted surfaces]{Polynomial parametrization of some interesting knotted surfaces}

\author[L. H. Kauffman]{Louis H. Kauffman}
\address{Department of Mathematics, Statistics and Computer Science \\ 
	University of Illinois at Chicago \\ 
	851 South Morgan Street\\
	Chicago, IL, 60607-7045}
\email{kauffman@uic.edu}

\author[T. MAHATO]{Tumpa Mahato}
\address{Department of Mathematics\\Indian Institute of Science Education and Research Pune, Dr. Homi Bhaba Road, Pune, Maharashtra, 411008, India.}
\email{2mpa.m80@gmail.com\\ tumpa.mahato@students.iiserpune.ac.in}

\author[R. Mishra]{Rama Mishra}
\address{Department of Mathematics\\Indian Institute of Science Education and Research Pune, Dr. Homi Bhaba Road, Pune, Maharashtra, 411008, India.}
\email{r.mishra@iiserpune.ac.in}

\makeatletter
\@namedef{subjclassname@2020}{%
	\textup{2020} Mathematics Subject Classification}
\makeatother


\begin{abstract}
	We discuss methods to construct a polynomial parametrization of some interesting knotted surfaces (knotted spheres, knotted tori and knotted planes) and provide examples.
\end{abstract}
\maketitle

\section{Introduction} 

Knots have a rich and long history and have been used for various activities like fishing, hunting, and building. The knots that we see are one dimensional objects inside our three dimensional space. To ensure a most simple form of a knot, one simplifies it without cutting it and mathematicians call it {\it ambient isotopy} \cite{Colin}. Many topological and geometric properties of a knot can be inferred if one can produce a parametrization 
of the knot using simple functions. Knot parametrizations are studied for {\it classical knots} by many mathematicians \cite{Shastri, Kau98, PraMis, Dun}.  Similar to a rope getting knotted, even surfaces (such as plane, sphere, torus) can get knotted. Surface knots are seen in computer graphics, physics and material science. In material science and textile, microscale   knots \cite{Sherif} are example of surface knots.
It is believed that these knotted surfaces provide stability and durability to the structure \cite{Moestopo}. Thus interesting knotted surfaces may help material scientists  design different and aesthetically  elegant structures.  Knotted surfaces are difficult to visualize due to their complex topology.
Thus if one can represent these knotted surfaces by nice parametric equations, one is able to judge its geometric properties that in turn provides an information on the strength it can create to the material it is used for.

In this paper, we have provided examples of knotted spheres, knotted tori and knotted planes. Our examples rely on constructing these surfaces with the help of one dimensional knots such as classical knots,  welded knots and long knots. We provide mathematical formulation to capture these constructions and use the knot parametrization of classical and long knots to parametrize these knotted surfaces. Thus, we can provide plenty of examples to be utilized by the scientists and engineers for their practical purposes.

This paper is organized as follows. Section \ref{sec:surface} includes the basic definitions required to study surface knots. Section \ref{sec:sphere} is divided into two subsections, Section \ref{sec:spin} and Section \ref{sec:twistspin} defines two constructions, spinning and twist spinning respectively for obtaining 2-knots from classical knots and discuss their parametrizations. In 
Section \ref{sec:ribbon}, we provide a brief exposition of the welded knots and its relationship with ribbon torus knots and then discuss the method to find a polynomial parametrization of the ribbon torus knots using that relationship. In Section \ref{sec:plane}, we discuss methods of constructing knotted planes from long classical knots. For each construction we have included the images generated from {\it Mathematica}.

\section{Basics on surface knots}\label{sec:surface}

  To study surface knots rigorously, we first recall several foundational concepts related to embeddings of manifolds. These definitions are essential for understanding the structure and classification of surfaces embedded in four-dimensional space.

\begin{definition}
	
	Let $N$ be an $n$-dimensional manifold embedded in an $m$-dimensional manifold $M$. We call an embedding $f:N \rightarrow M$ to be {\it proper} if every compact subset $K \subset f(N)$ has a compact preimage $f^{-1}(K)=\{x \in N \vert f(x) \in K\}$. In other words, the manifold $f(N)$ is proper if $\partial f(N)= f(N) \cap \partial M$.
\end{definition}

 To ensure that surface embeddings behave nicely in small neighbourhoods, we need a notion of local flatness.
\begin{definition}
	Let $f:N\to M$ be a proper embedding of smooth manifolds. This embedding is said to be 
  {\it  locally flat}  at a point $x \in M$ if there exists a regular neighborhood $U$ of $x$ in $M$ such that the pair $(U,U \cap f(N))$ is homeomorphic to the standard ball pair $(D^{m},D^{n})$, where $D^{k}$ is the unit ball in $\R^{k}$, centred at the origin.
	We say that $f$ is {\it locally flat} if it is locally flat at every point of $M$.
\end{definition}

 Not all proper embeddings are locally flat. Examples of non-locally flat embeddings can be found in \cite{Kam}, highlighting the necessity of this condition in knot theory.

With these foundational notions established, we now introduce the definitions of surface knots and surface links.
\begin{definition}
	A {\it surface link} is a proper, locally flat embedding of a closed surface $F$ in $\R^{4}$ or $S^{4}$.  If $F$ is a connected surface then it is called a {\it surface knot}. When $F=S^{2}$, it is called a {\it 2-knot} or {\it $2$-dimensional knot}.
	
\end{definition}

 Surface knots can also appear in non-compact forms, leading to the notion of long surface knots.

\begin{definition}
	A \emph{long $2$-knot} is a proper, locally flat embedding of $\R^2$ into $\R^4$ that is standard outside a compact region.
\end{definition}

\begin{remark}
	Every $2$-knot in $S^4$ naturally gives rise to a long $2$-knot in $\R^4$ by removing a point from the embedded $S^2$ and its image from $S^4$. This converts the compact embedding into a non-compact one, effectively yielding a long version of the $2$-knot.
\end{remark}
    Now, we recall an interesting class of knotted surfaces, called ribbon surface knots.
\begin{definition}
	Let $M$ be an compact $3$-manifold immersed in $\R^{4}$ and $f:M \rightarrow \R^{4}$ be the immersion map. A connected component $D$ of the double point set is called a {\it ribbon singularity} if the preimage $f^{-1}(D)$ is the union of $2$-disks $D_{1}$ and $D_{2}$ such that $D_{1}$ is a properly embedded 2-disk in $M$ and $D_{2}$ is an embedded $2$-disk in the interior of $M$.
\end{definition}

\begin{definition}\label{def:ribbon}
	Let $\Delta= \Delta_{1} \sqcup \cdots \sqcup \Delta_{k}$ and $\mathcal{B}= \mathcal{B}_{1} \sqcup \cdots \sqcup \mathcal{B}_{l}$ be two disjoint collections of $3$-balls embedded in $\R^{4}$, with each component of $\Delta$ and $\mathcal{B}$ is called a base and a band respectively. We  parametrize each band by $h_{i}: D^{2} \times [0,1] \rightarrow \R^{4}$ such that 
	\begin{enumerate}[itemsep=0pt]
		\item For each $i=1,2,\cdots,l$,
			\[\partial \Delta \cap h_{i}(D^{2}\times [0,1])= h_{i}(D^{2} \times \partial[0,1]), \quad \text{and}\]
			\item If $h_{i}(D^{2} \times [0,1])$ intersects a base, then it intersects this base in meridional 2-disks that are prpoerly embedded in the band and contained in the interior of the base:
		\[h_{i}(D^{2} \times [0,1]) \cap \Delta_{j} = h_{i}(D^{2} \times \{t\})\] 
		for $t \in (0,1)$. And all these intersections are ribbon singularities (Fig. \ref{fig:ribbon presentation}).
	\end{enumerate} 
	Now the surface given by
	$F= \LS \partial \Delta \, \setminus \, \cup_{i=1}^{l}h_{i}(D^{2} \times [0,1])\,\RS \bigcup \, \cup_{i=1}^{l}h_{i}(\partial D^{2} \times [0,1]), $
	is called a {\it ribbon surface knot} (Fig. \ref{fig:ribbon knot}).
	\begin{itemize}
		\item If $l=k-1$ then $F$ is homeomorphic to a $2$-sphere and $F$ is called a {\it ribbon 2-knot} or a {\it ribbon sphere}.
		\item If $F$ is homeomorphic to a torus then we call it a {\it ribbon torus-knot}.
		\item The pair $(\Delta,\mathcal{B})$ is called a {\it ribbon presentation} of $F$ with $l$ fusion bands.
	\end{itemize}
\end{definition}
The intersections of type (1) and (2) are depicted in Fig. \ref{fig:ribbon presentation} for classical case (bottom) and surface case (top). Fig. \ref{fig:ribbon knot} shows a broken surface diagram of a ribbon surface knot where parts of the surface is removed to indicate the over/under information similar to classical knot theory.
\begin{figure}[H]
	\centering
	\includegraphics[width=0.5\linewidth]{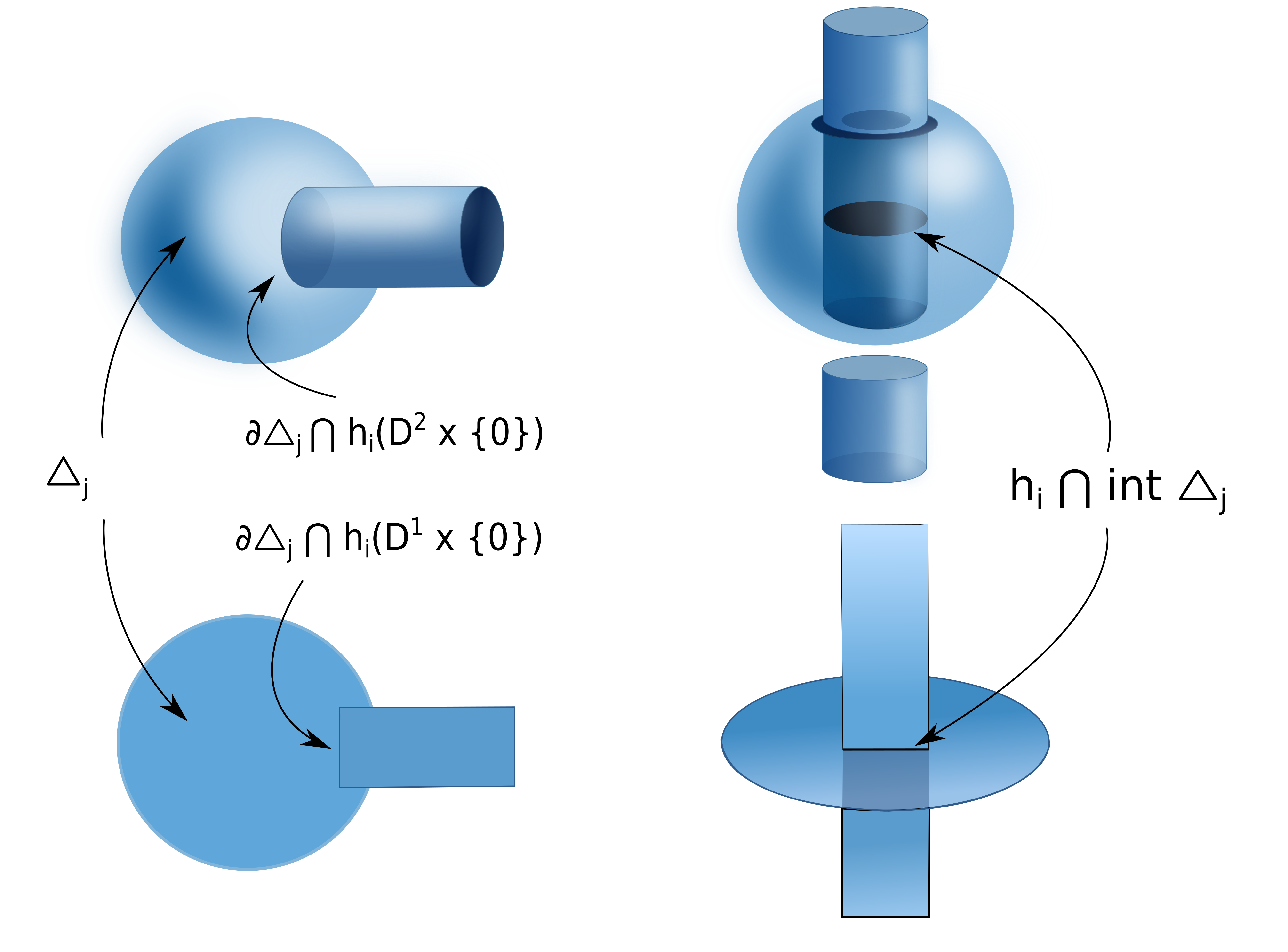}
	\caption{Ribbon presentation in dimension one and two}
	\label{fig:ribbon presentation}
\end{figure}
\begin{figure}[H]
	\centering
	\includegraphics[width=0.4\linewidth]{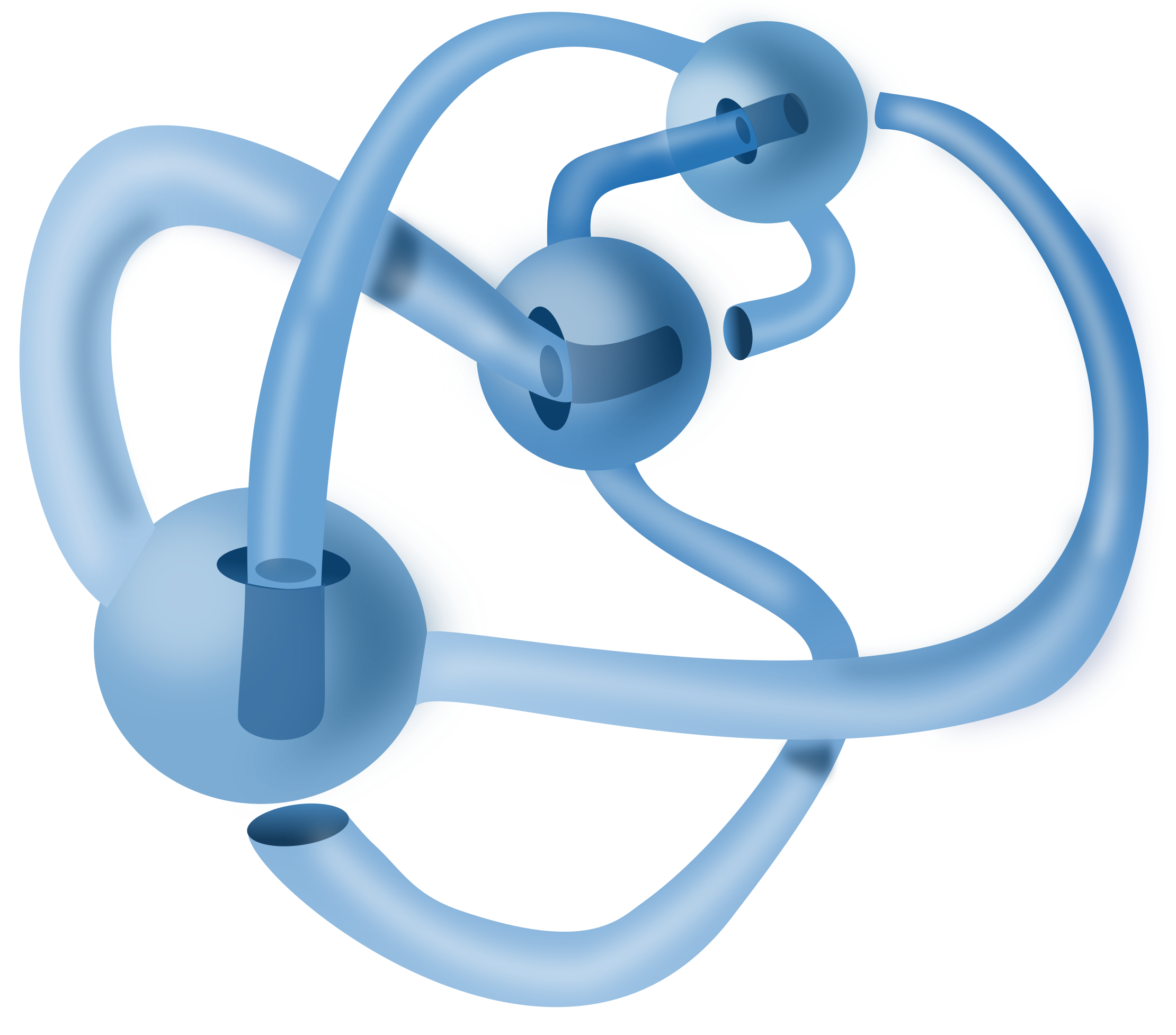}
	\caption{A broken surface diagram of a ribbon surface knot}
	\label{fig:ribbon knot}
\end{figure}

For a detailed exposition of surface knots and their topological subtleties, we refer the reader to \cite{CarSai,Glu,Kam}. 

In the upcoming sections we will discuss knotted spheres, knotted tori and knotted planes.  Our main focus is to provide parametric equations, so in each case, we have attempted only those classes of knotted surfaces which arise naturally from the one dimensional knots so that their parametrization can be used.

\section{ Parametrizing  knotted spheres  } \label{sec:sphere}

A two dimensional sphere can get knotted inside a four dimensional space $\mathbb{R}^4$ or $S^4$. One can study different isotopy classes with the help of invariants.  These $2$-dimensional knots are also known as $2$-knots. Since these knots are in $\mathbb{R}^4$  it is difficult to visualize them. Moreover, there is no uniform method to enumerate them and study them in order.  Simplest methods to visualize a knotted $S^2$ is through it projection into some three dimensional space or to look at the {\it movie } \cite{CarSai}. Our goal is to produce parametric equations that represent interesting knotted spheres. It is known that certain  operations  (spinning, twist spinning, roll spinning \cite{Fox}) performed on a classical knot inside $\mathbb{R}^3$ result into a knotted sphere inside $\mathbb{R}^4$. Thus with a parametric equations of a classical knot in hand, if we can formulate these operations mathematically, we obtain parametrization of such knotted spheres. We have succeeded in constructing spinning and $d$-twist ($d\geq 0$) spinning operations and the knots obtained are called {\it spun $2$-knots} and {\it $d$-twist spun $2$-knots}. We describe these constructions in Section 3.1 and Section 3.2 respectively.

 \subsection{Spun $2$-knots}\label{sec:spin}

E.~Artin \cite{Art} introduced a way to construct $2$-knots by spinning a knotted arc embedded in the half-space $\R^{3}_{+}$ around $\R^{2}$. The resulting $2$-knots are called {\it spun 2-knots}. Spun knots represent the simplest class of $2$-knots.
  Their construction is described as follows. In $\R^{4}$, consider the upper half space $\R^{3}_{+}=\LC(x, y, z, 0): z\geq 0 \RC$
with the boundary
$\partial \R^{3}_{+}=\R^{2}=\LC(x, y, 0, 0)\RC.$
Now, locus of a point $x=(x, y, z, 0)$ in $\R^{3}_{+}$, rotated in 4-space about $\R^{2}$ is given by
\[\LC(x, y, z \,\cos \theta, z \,\sin \theta) \MV  0 \leq\theta\leq 2\pi \RC.\]
	\begin{figure}[H]
	\centering
	\includegraphics[width=0.23\textwidth]{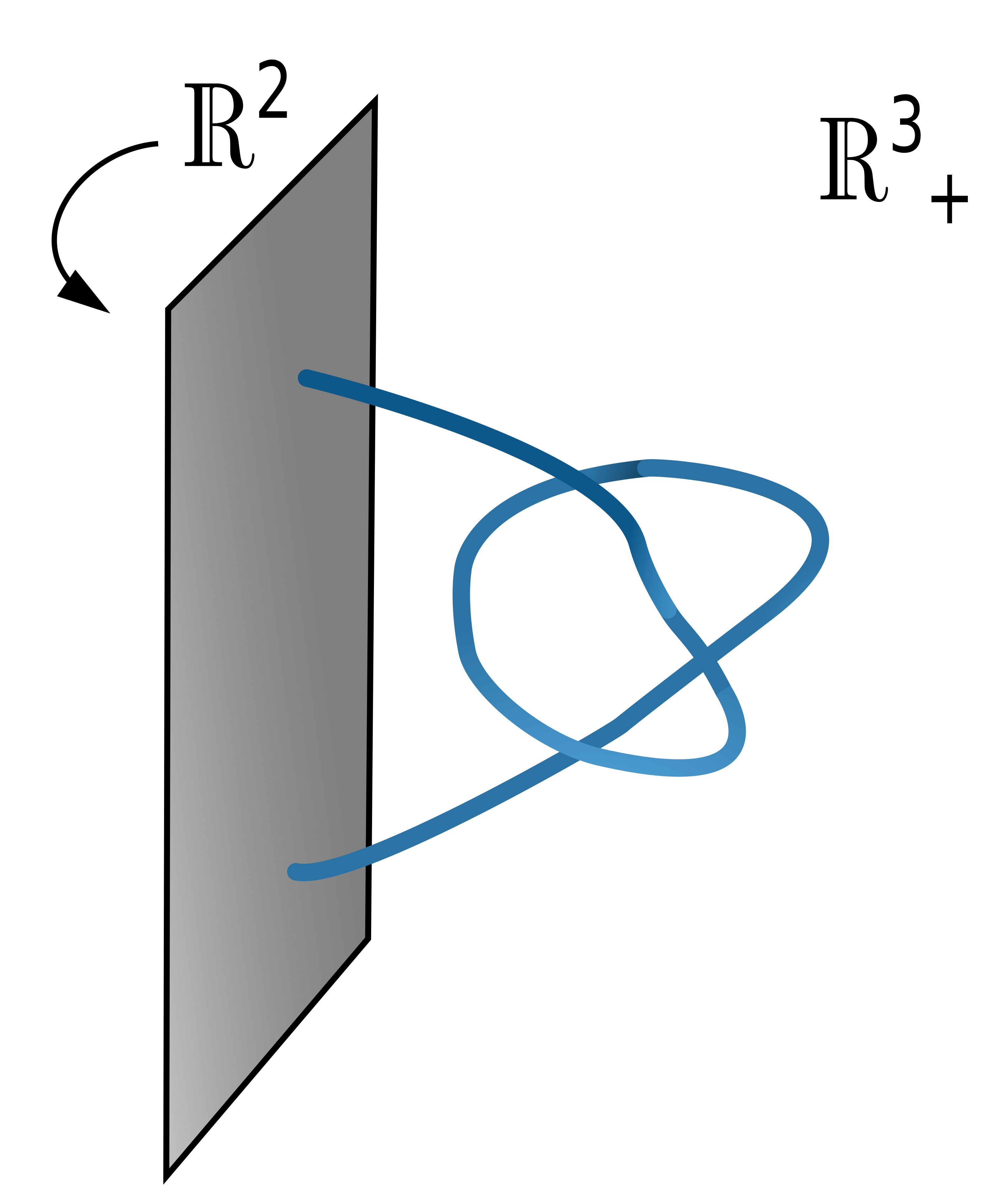}
	\caption{Spinning}
	\label{fig: arc in upper space}
\end{figure}
To construct a spun 2-knot, we choose a properly, locally flatly embedded arc $k$  in $\R^{3}_{+}$ i.e $k$ is embedded in $\R^{3}_{+}$ locally flatly and intersects $\partial \R^{3}_{+}$ transversely only at the endpoints (Fig. \ref{fig: arc in upper space}). Then we spin $\R^{3}_{+}$ along $\R^{2}$ by $360^{o}$ and the continuous locus of the arc $k$ produces the spun 2-knot given by
\[\LC(x, y, z \,\cos \theta, z \,\sin \theta) \MV (x, y, z)\in k, 0 \leq\theta\leq 2\pi \RC.\]

\begin{theorem}\label{th:spun}
	Given a classical knot $K$, there exist polynomials $f(s,t)$, 
	$g(s,t), h(s,t)$ and $ p(s,t)$ in two variables 
	$s$ and $t$ such that for some interval $[a,b]$ the image of  $[a,b]\times [0,2\pi]$ under the map $\phi: \R^2 \to \R^4$ defined by $$\phi ((s,t))= \LP f(s,t), g(s,t), h(s,t), p(s,t) \RP$$ is isotopic to the spun of $K$.
\end{theorem}
\begin{proof}
	Let us choose a polynomial representation $\phi$ of $K$ such that  $\phi: [a,b] \rightarrow \R^{3}_{+}$ of the knotted arc $k$ is given by 
	$\phi(t) \rightarrow \LP f(t), g(t), h(t)  \RP$
	such that $h(a)=0=h(b)$ and $h(t)>0$ for $t \in (a,b)$. Now the spinning construction gives a map $F:[a,b]\times [0, 2\pi]\to \R^4$ defined by
	\[F(s,t)=\LP f(t),\, g(t),\,\, h(t)\,\cos s,\, h(t)\,\sin s \RP\] 
	that represents $spun(K)$.
	Using a point set topology argument, one can prove that the image of $F$ in $\R^4$ is homeomorphic to $S^2$. Now, we replace the trigonometric functions $\cos s$ and $\sin s$ with their Chebyshev approximations inside the interval $[0, 2\pi]$   \cite{Abu}. Let the Chebyshev approximations of $\cos s$ and $\sin s$ inside the interval $[0,2\pi]$ be denoted by $C(s)$ and $S(s)$, respectively. Then by choosing $f(s,t)=f(t)$, $g(s,t)=g(t)$, $h(s,t)=h(t)\,C(s)$ and $p(s,t)=h(t)\,S(s)$, we obtain a polynomial  parametrization of $spun(K)$. This completes the proof.
\end{proof}

\begin{example}[The spun trefoil knot] Let us take the following polynomial representation of the long trefoil knot \cite{Shastri} given by
	\[ \LC \LP t^3-3t, t^5-10t , t^4-4t^2  \RP \MV t \in \R  \RC .\] 
	We change the function $(t^4-4t^2)$ to $(-t^4+4t^2+3)$
	which has real roots $t= -2.1554 $ and $t=2.1554 $. Therefore, the knotted arc $k$ (Fig. \ref{fig: long trefoil knot}) is given by
	$$ \LC \LP t^3-3t, t^5-10t , -t^4+4t^2 +3 \RP \MV t \in [-2.1554,2.1544] \RC .$$ 
	and the spun trefoil knot is given by
	\[\LC \LP t^3-3t, t^5-10t , (-t^4+4t^2 +3)\,C(s), (-t^4+4t^2 +3)\,S(s) \RP \RC\]
	for $t \in [-2.1554,2.1544], s\in [0, 2\pi]$. For the final polynomial parametrization we take the Chebyshev approximations $C(s)$ and $S(s)$ of $\cos s$ and $\sin s$ in $[0,2\pi]$ respectively as follows:
       \begin{equation}\label{eq:chebcos}
		\begin{split}
			C(s) = \, -0.0000193235  s^8+0.000485652  s^7 -  0.00399024  s^6+0.0081095  s^5 
			 +0.0265068  s^4 &\\
			  +0.0163844  s^3-0.509175  s^2 
			 +0.00205416 s+0.999921, \qquad &
		\end{split}
		\end{equation}
		 \begin{equation}\label{eq:chebsin}
			\begin{split}
				S(s) = \,8.73651067430188 \times 10^{-19}  s^8+0.000144829  s^7 -0.00318496  s^6 
			+0.0220637  s^5 &\\
			-0.0322337  s^4 -0.125592  s^3-0.0257364  s^2 
			 +1.00614 s-0.000238495.\qquad& 
			\end{split}
		\end{equation}

	{\it Mathematica} plot for a projection of the spun trefoil are shown in Fig. \ref{fig:Spun trefoil projection}.
	In Fig. \ref{fig:a hyperplane section}, an inside view is also given where it can be seen that the knotted arc does not deform while spinning around $\R^{2}$.
\end{example}
	\begin{figure}[H]
		\centering
		\includegraphics[width=0.65\textwidth,,height=0.25\textwidth]{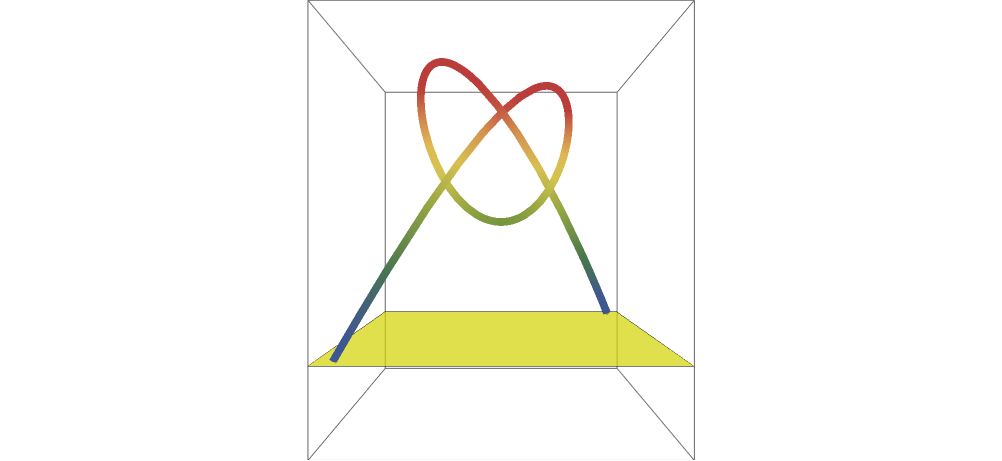}
		\caption{Knotted  arc of the long trefoil knot}
		\label{fig: long trefoil knot}
	\end{figure}
	\begin{figure}[H]
		\centering
		\begin{subfigure}[c]{0.4\linewidth}
			\centering
			\includegraphics[width=1.2\textwidth]{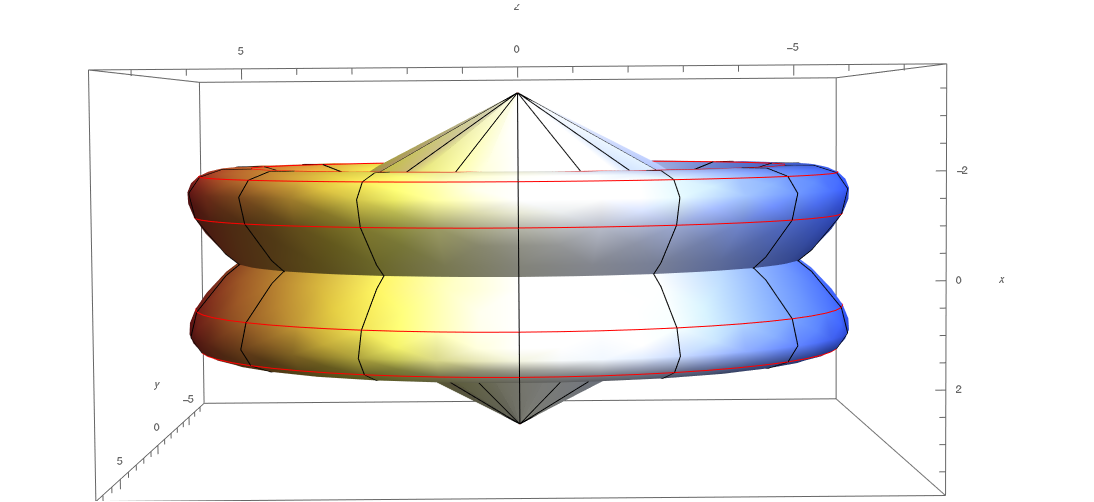}
			\caption{Projection on a hyperplane}
			\label{fig:Spun trefoil projection}
		\end{subfigure}
		\hspace{1.5cm}
		\begin{subfigure}[c]{0.4\linewidth}
			\centering
			\includegraphics[width=\textwidth]{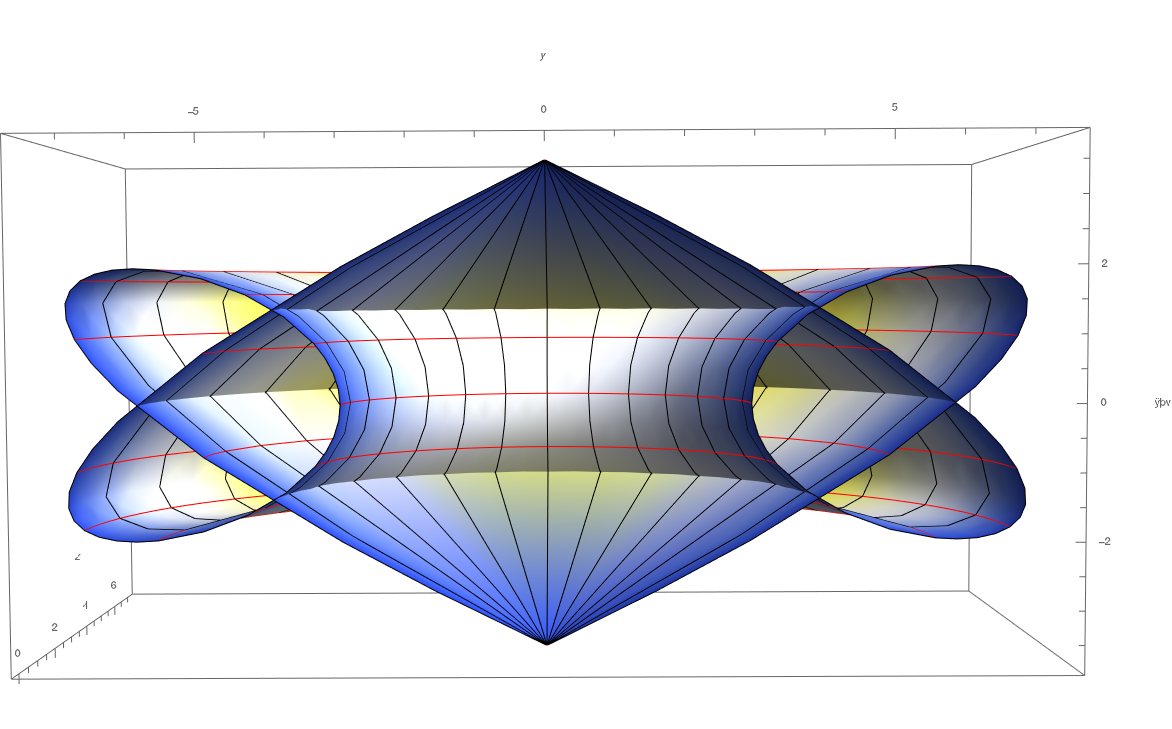}
			\caption{An inside view }
			\label{fig:a hyperplane section}
		\end{subfigure}
		\caption{The spun trefoil knot}
		\label{fig:spun trefoil}
	\end{figure}

\begin{example}[The spun figure-eight knot]
	Using the  parametrization for the long figure eight knot given in \cite{ANB} (Fig. \ref{fig: long figure eight knot}) and following the same procedure as above, we obtain a polynomial  parametrization of the spun figure eight knot given
	by $$\LC \LP f(t), g(t), h(t) \, C(s), h(t)\, S(s)\, \RP \MV t\in [-3.7934, 3.7934], s\in [0,2\pi]\RC$$
\begin{align*}
		\text{	where}\qquad f(t) &=\frac{2}{5}  \LP t^2-7 \RP    \LP t^2-10 \RP   t, \\
		g(t) &= \frac{1}{10} t  \LP t^2-4 \RP    \LP t^2-9 \RP    \LP t^2-12 \RP  , \\
		h(t) &=(20-13 t^2-t^4).
	\end{align*}

	and $C(s)$, $S(s)$ are the Chebyshev approximations of $\cos s$ (Equation~\ref{eq:chebcos}) and $\sin s$ (Equation  \ref{eq:chebsin}) respectively. Fig. \ref{fig: spun figure eight} shows the {\it Mathematica} plot of a projection of the spun figure eight knot on \textit{XZW-plane}.
\end{example}

\begin{figure}[H]
	\centering
	\includegraphics[width=0.38\textwidth]{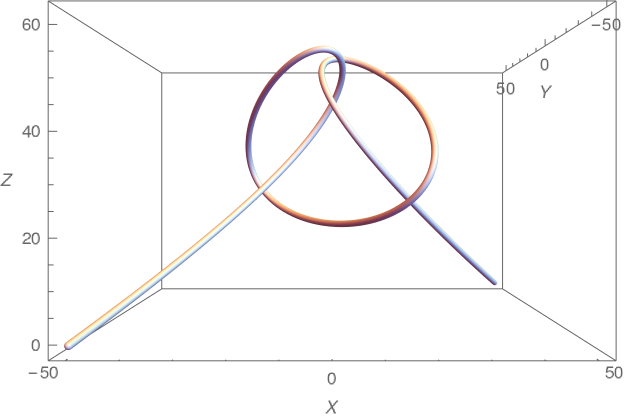}
	\caption{Knotted  arc of long figure eight knot}
	\label{fig: long figure eight knot}
\end{figure}
	\begin{figure}[H]
		\centering
		\subfloat[Projection on a hyperplane]{\includegraphics[width=0.44\textwidth]{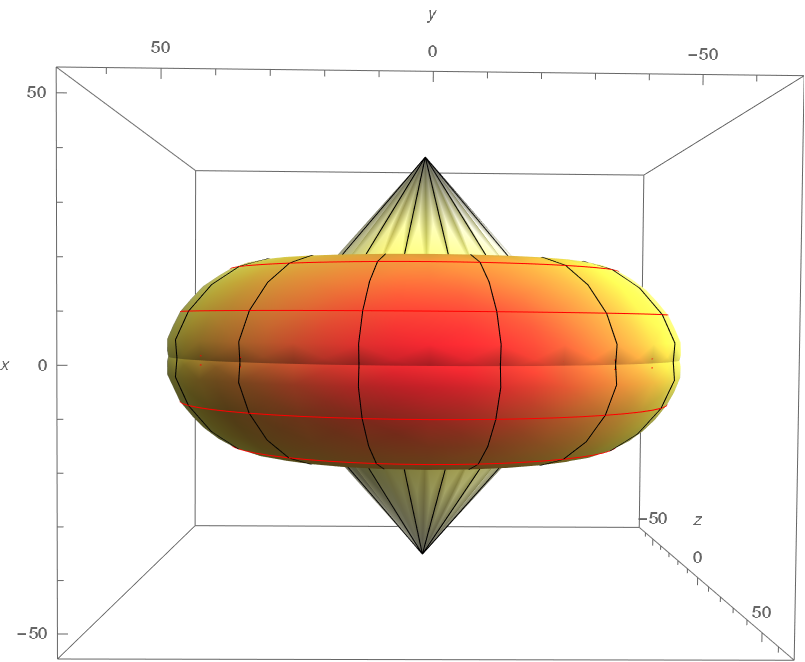}}
		\quad
		\subfloat[Inside view of the projection]{	\includegraphics[width=0.45\textwidth]{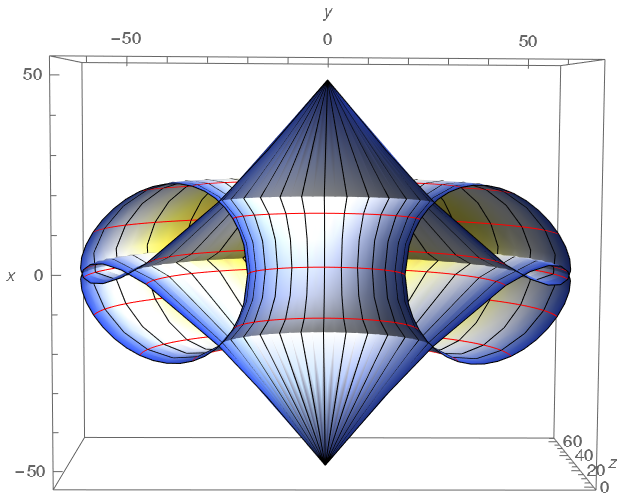}}
		\caption{The spun figure eight knot}
		\label{fig: spun figure eight}
	\end{figure}
	
	\begin{remark}
		Thus for every classical knot $K$, we will be able to construct a spun $2$-knot parametrized by polynomial functions. 
	\end{remark}

\subsection{ $d$-twist spun $2$-knots}\label{sec:twistspin}

E. C. Zeeman generalized Artin's spinning construction to {\it twist spinning} in 1965 \cite{Zee65}. In this case, we imagine the knotted part of $k$ inside a 3-ball as in Fig. \ref{fig:twistspinning}. Now, to include twisting in the spinning construction we rotate the ball $d$ times around its own axis while rotating $\R^{3}_{+}$ around $\R^{2}$ once. Position of the knotted arc after $d$ twists should match its initial position after completing the rotation around $\R^{2}$. This way we get a 2-sphere in $\R^{4}$, called {\it $d$-twist spun 2-knot}. By definition, for $d=0$ we get a spun knot.
 \begin{figure}[H]
	\begin{center}
		\includegraphics[width=0.25\linewidth]{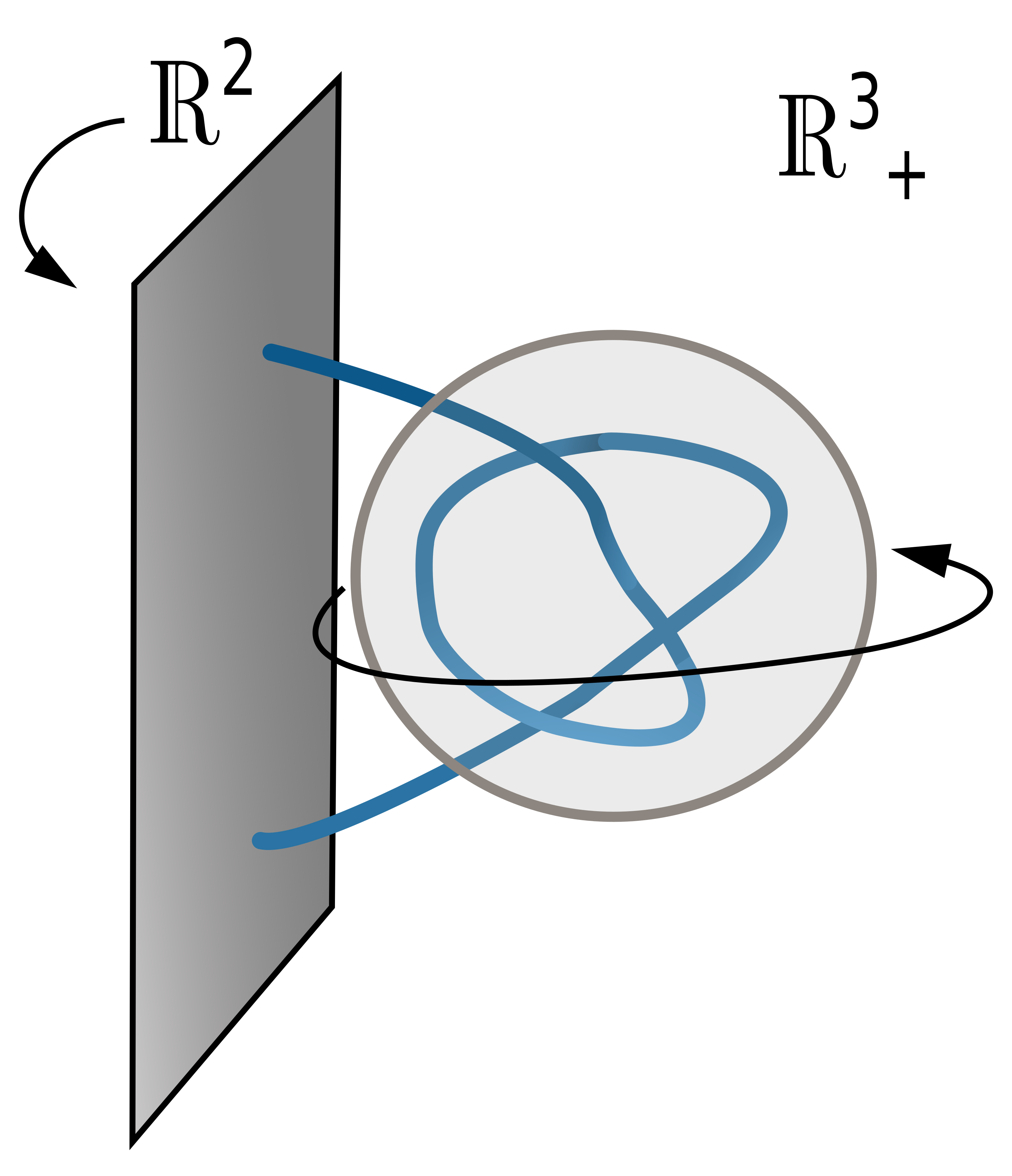}
	\end{center}
	\caption{Twist spinning}
	\label{fig:twistspinning}
\end{figure}
\begin{theorem}
	Given a classical knot $K$, there exist polynomials $f(t,\theta)$, 
	$g(t,\theta), h(t,\theta)$ and $ p(t,\theta)$ in two variables 
	$t$ and $\theta$ such that for some interval $[a,b]$, the image of  $[a,b]\times [0,2\pi]$ under the map $\phi: \R^2 \to \R^4$ defined by $$\phi (t,\theta)= \LP \bar{f}(t,\theta), \bar{g}(t,\theta), \bar{h}(t,\theta), \bar{p}(t,\theta)\RP $$ is isotopic to the $d$-twist spun of $K$. 
\end{theorem}
\begin{proof}
	We start with the  parametrization $\phi: [a,b] \rightarrow \R^{3}_{+}$ of the knotted arc $k$, given by 
	\[\phi(t) \rightarrow \LP f(t), g(t), h(t)  \RP\] 
	where $h(a)=0=h(b)$ and $h(t)>0$ for $t \in (a,b)$.
	The endpoints of $k$ on the boundary $\R^{2}$ is given by $(f(a),g(a),0)$ and $(f(b),g(b),0)$. Now we find an interval $[a',b'] \subset [a,b]$ where all the crossings of $K$ lie.
	
	In the twist spinning construction, the knotted part of the arc is contained in a 3-ball in $\R^{3}_{+}$ and is rotated around the axis of the ball. To achieve this first we choose the axis of rotation as a line segment $PQ$ parallel to the $XY$ plane joining two points on the knotted arc, say $P:=(f(t_{1}),g(t_{1}),h(t_{1}))$ and $Q:=(f(t_{2}),g(t_{2}),h(t_{2}))$, where $[a',b'] \subset [t_{1},t_{2}]$ and $h(t_{1})=h(t_{2})=c$ (Fig. \ref{Rotating about $PQ$}). The value of $c$ is chosen so that the knotted part of the arc does not intersect the $XY$ plane while rotating around $PQ$. By Rodrigues' rotation formula \cite{Rod}, the rotation around $PQ$ is represented by the following matrix.
	\[\mathbf{R}'=\mathbf{T_{c}*R*T_{c}^{-1}}=\mathbf{T_{c}*R*T_{-c}},\]
	where $\mathbf{T_{c}}$ is the translation matrix along Z axis which will send $(x,y,z)$ to $(x,y,z+c)$ and $\mathbf{R}$ gives the rotation around the line $P'Q'$ on $XY$ plane, parallel to $PQ$, joining the points $P'=(f(t_{1}),g(t_{1}),0)$ and $Q'=(f(t_{2}),g(t_{2}),0)$ on the knotted arc.
	
	\begin{figure}[H]
		\centering
		\includegraphics[width=0.35\textwidth]{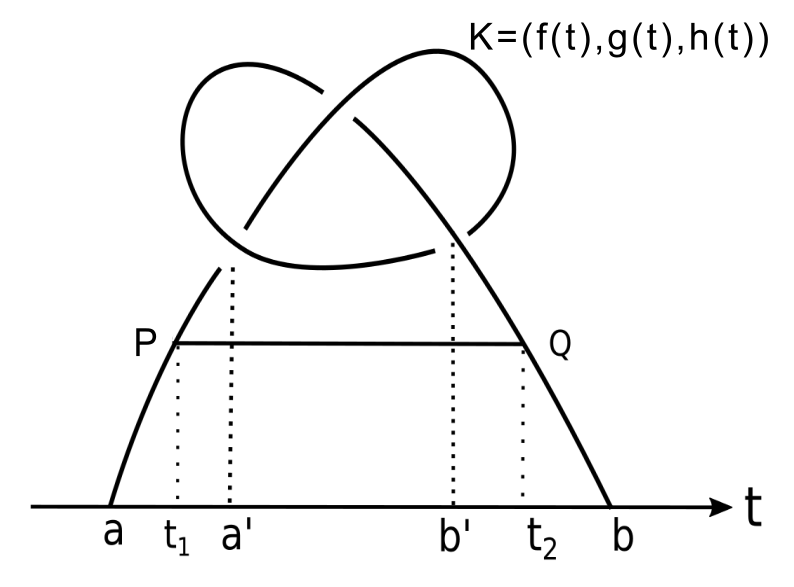}
		\caption{Rotating about $PQ$}
		\label{Rotating about $PQ$}
	\end{figure}
	The rotation matrix $\mathbf{R}$ is defined as follows:
	If $\mathbf{v}$ is a vector in $\R^{3}$ and $\mathbf{k}$ is a unit vector describing an axis of rotation around which $\mathbf{v}$ rotates by an angle $\phi$ according to the right-hand rule, the Rodrigues formula for the rotated vector $\mathbf{v_{rot}}$ is given by
	\[\mathbf{v}_{rot}=\mathbf{v}\,\cos\phi +(\mathbf{k} \times \mathbf{v})\sin\phi+\mathbf{k}\cdot(\mathbf{k}\cdot \mathbf{v})(1-\cos\phi).\]
	In this case, $\mathbf{k}$ is along the line joining $(f(t_{1}),g(t_{1}),0)$ and $(f(t_{2}),g(t_{2}),0)$.
	So,\[\mathbf{k}=\Big(\frac{f(t_{2})-f(t_{1})}{N},\frac{g(t_{2})-g(t_{1})}{N},0\Big),\]
	where $N=\sqrt{(f(t_{2})-f(t_{1}))^{2}+(g(t_{2})-g(t_{1}))^{2}}.$
	
	For simplification, let us denote 
	\[f_{21}:=f(t_{2})-f(t_{1})\] 
	\[g_{21}:=g(t_{2})-g(t_{1}).\]
	
	Then the rotation matrix through an angle $\phi$ counterclockwise around the axis $\mathbf{k}$ is
	\[\mathbf{R=I+\sin\phi \; K+(1-\cos\phi)\;K^{2}},\]
	where
	\[ \mathbf{K}=
	\begin{pmatrix}
		0 & -k_{z} & k_{y} \\
		k_{z} &0 & -k_{x} \\
		-k_{y} & k_{x} & 0
	\end{pmatrix}
	=\begin{pmatrix}
		0 & 0 &\dfrac{g_{21}}{N} \\[0.5cm]
		0 & 0 & -\dfrac{f_{21}}{N} \\[0.5cm]
		-\dfrac{g_{21}}{N} & \dfrac{f_{21}}{N} & 0
	\end{pmatrix}.\]
	and the rotation matrix around $PQ$ is given by,
	\[\mathbf{R}'=\mathbf{T_{c}*R*T_{c}^{-1}}=\mathbf{T_{c}*R*T_{-c}},\]
	where
	\[\mathbf{R}=
	\begin{pmatrix}
		\dfrac{f_{21}^{2}\;+\;g_{21}^{2}\;\cos\phi}{N^{2}} & \dfrac{f_{21}\;g_{21}\;(1-\cos\phi)}{N^{2}} & \dfrac{g_{21}\;\sin\phi}{N} & 0\\[0.5cm]
		\dfrac{f_{21}\;g_{21}\;(1-\cos\phi)}{N^{2}} &   \dfrac{f_{21}^{2}\;\cos\phi\;+\;g_{21}^{2}}{N^{2}}  & -\dfrac{f_{21}\;\sin\phi}{N}& 0\\[0.5cm]
		-\dfrac{g_{21}\;\sin\phi}{N} & \dfrac{f_{21}\;\sin\phi}{N}&\cos\phi & 0 \\[0.5cm]
		0 & 0 & 0 & 1
	\end{pmatrix}\]
	and $\mathbf{T_{c}}$ is the translation matrix along Z-axis which sends $(x,y,z)$ to $(x,y,z+c)$, given by \[T_{c}=
	\begin{pmatrix}
		1 & 0 & 0 & 0 \\
		0& 1 & 0 & 0\\
		0 & 0 & 1 & c \\
		0 & 0 & 0 & 1
	\end{pmatrix}.\]
	Then,
	\[\mathbf{R}'=
	\begin{pmatrix}
		\dfrac{f_{21}^{2}\;+\;g_{21}^{2}\;\cos\phi}{N^{2}} & \dfrac{f_{21}\;g_{21}\;(1-\cos\phi)}{N^{2}} & \dfrac{g_{21}\;\sin\phi}{N} & -\dfrac{c\;g_{21}\;\sin\phi}{N}\\[0.5cm]
		\dfrac{f_{21}\;g_{21}\;(1-\cos\phi)}{N^{2}} &   \dfrac{f_{21}^{2}\;\cos\phi\;+\;g_{21}^{2}}{N^{2}}  & -\dfrac{f_{21}\;\sin\phi}{N} & \dfrac{c\;f_{21}\;\sin\phi}{N}\\[0.5cm]
		-\dfrac{g_{21}\;\sin\phi}{N} & \dfrac{f_{21}\;\sin\phi}{N}&\cos\phi& -c\;\cos\phi+c \\[0.5cm]
		0 &0 &0 & 1 
	\end{pmatrix}.\]
	Thus, the rotation around $PQ$ is given by the  parametrization :
	\begin{equation}
		(t,\phi) \longrightarrow  \LC \LP f'(t,\phi),g'(t,\phi),h'(t,\phi) \RP\,\middle|\,
		\begin{aligned}
			a \leq &t \leq b \\
			0 \leq &\phi < 2\pi
		\end{aligned}\, \RC 
	\end{equation}
	where
	{\small \begin{align}
			f'(t,\phi)&=\frac{\LP f_{21}^{2}\;+\;g_{21}^{2}\;\cos\phi \RP\,\mathbf{f(t)}+f_{21}\;g_{21}\;(1-\cos\phi)\,\mathbf{g(t)}+N\;g_{21}\,\sin\phi\,(\mathbf{h(t)}-c)}{N^2}\\
			g'(t,\phi)&=\frac{f_{21}\;g_{21}\;(1-\cos\phi)\,\mathbf{f(t)}+\LP f_{21}^{2}\;\cos\phi\;+\;g_{21}^{2} \RP\,\mathbf{g(t)}+N\;g_{21}\,\sin\phi\,(c-\mathbf{h(t)})}{N^{2}}\\
			h'(t,\phi)&=\frac{-g_{21}\;\sin\phi\;\mathbf{f(t)}+f_{21}\sin\phi\;\mathbf{g(t)}+N
				\,\cos\phi\,\mathbf{h(t)}+N\,c\,(1-\cos\phi)}{N}
	\end{align} }
	
	Now, in the twist-spinning construction, there is no rotation outside the $3$-ball that contains the knotted part. This is same as not rotating the knotted arc outside the endpoints of $PQ$. We achieve this by combining the rotation functions with a bump function.

	We have, $ [a,b]= [a,t_{1}] \cup [t_{1},a'] \cup [a',b'] \cup [b',t_{2}] \cup [t_{2},b]$ (Fig. \ref{Rotating about $PQ$}), where
	\begin{enumerate}
		\item $[a',b']$ contains all the crossings of the knot.
		\item $t_{1},t_{2}$ corresponds to the endpoints of $P$ and $Q$. So, $[a,t_{1}]$ and $[t_{2},b]$ correspond to the parts of the arc outside the axis $PQ$. 
	\end{enumerate}
		\begin{figure}[H]
		\centering
		\includegraphics[width=0.3\textwidth,height=0.2\textwidth]{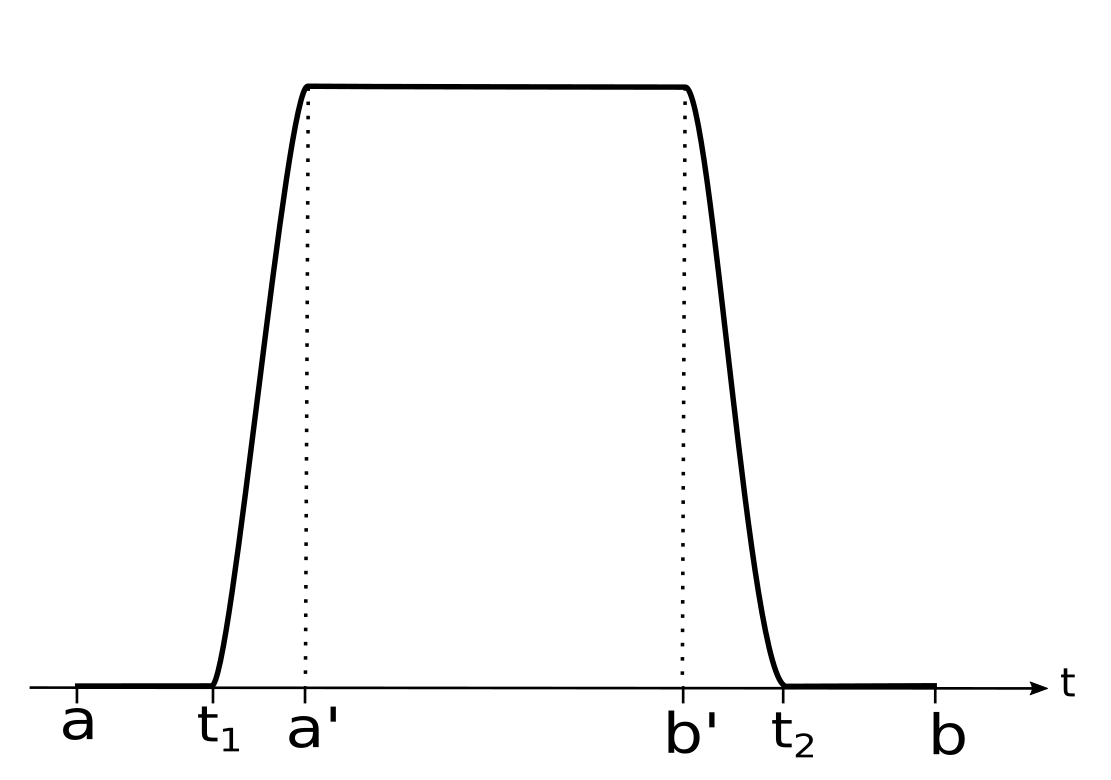}
		\caption{Bump function B(t). }
		\label{Bump function}
	\end{figure}
	Therefore, we define a bump function which takes the value $1$ in $[a',b']$ and $0$ outside $[t_{1}, t_{2}]$. First, consider the function 
	\[F(t)=\begin{cases}
		e^{-\frac{1}{t}}\,, & t > 0 \\
		0 \,, & t \leq 0.
	\end{cases}\]
	Then define the bump function, $B(t)$ as follows:
	\[B(t)=\frac{F(d_{1}-t^{2})}{F(t^{2}-d_{2})+F(d_{1}-t^{2})},\]
	where $d_{1},d_{2} \in \R^{+} $ are chosen in such a way that 
	\[B(t)=\begin{cases}
		1\,, & t \in [a',b'] \\
		0\,, & t \in [a,t_{1}] \cup [t_{2},b]\\
		\in (0,1)\,,& \text{otherwise}.
	\end{cases}.\] 
	See Fig. \ref{Bump function}.	Then, for $t \in [a,b]$, we define
	\begin{align*}
		\Tilde{f}(t,\phi)=& \,B(t)\,f'(t,\phi)+(1-B(t))\,f(t),\\
		\Tilde{g}(t,\phi)=& \,B(t)\,g'(t,\phi)+(1-B(t))\,g(t),\\
		\Tilde{h}(t,\phi)=& \,B(t)\,h'(t,\phi)+(1-B(t))\,h(t) .
	\end{align*}
	
	Now, the knotted arc is rotating $d$ times around the $PQ$ axis in $\R^{3}_{+}$ while rotating around the $XY$ plane in $\R^{4}$. Therefore, the rotation angle around $PQ$ is $d$ times the rotation angle around the $XY$ plane i.e. $\phi=d\,\theta.$
	
	Hence, the  parametrization for $d$-twist spun knot is given by 
	\begin{equation}
		(t,\theta) \longrightarrow  \LC \LP \Tilde{f}(t,d\theta)\,,\,\Tilde{g}(t,d\theta)\,,\,\Tilde{h}(t,d\theta)\,\cos\,\theta\,,\,\Tilde{h}(t,d\theta)\,\sin\,\theta  \RP \MV
		\begin{aligned}
			a \leq &t \leq b \\
			0 \leq &\theta < 2\pi
		\end{aligned}  \RC .
	\end{equation}

	Here, $\Tilde{f},\Tilde{g},\Tilde{h}$ are linear combinations of the functions $f,g,h$, cosine, sine and the bump function $B(t)$. Replacing cosine, sine and $B(t)$ with their corresponding Chebyshev polynomial approximations in $[0,2\pi]$, we obtain a polynomial  parametrization of $d$-twist spun knots given by 
	\begin{equation}
		\LC \LP \bar{f}(t,\theta)\,,\,\bar{g}(t,\theta)\,,\,\bar{h}(t,\theta)\,,\,\bar{p}(t,\theta)  \RP \MV a \leq t \leq b, \, 0 \leq \theta < 2\pi
		\RC .
	\end{equation} 
	This completes the proof.
\end{proof}

\begin{example}[$d$-twist spun trefoil]
	Start with a polynomial parametrization of the long trefoil knot (Fig. \ref{fig: trefoil with axis}), given by
		\begin{align*}
		x(t)&=t^3-3t \\
		y(t)&=t^5-10t \\
		z(t)&=-t^4+4t^2+16 .
	\end{align*}
Solving for $t_{1},t_{2}$ such that $x(t_{1})
	=x(t_{2})$ and $y(t_{1})
	=y(t_{2})$ we can get the crossing data and solving $z(t)=0$ will give the interval for the arc. Therefore, we have \begin{align*}
		[a,b]&=[-2.54404,2.54404],\\
		[a',b']&=[-1.946,1.946] ,\\
		P&=(f[-2.19],g[-2.19],h[-2.19]),\\
		Q&=(f[2.19],g[2.19],h[2.19]),\\
		c&=h[2.19].
	\end{align*}
		\begin{figure}[H]
		\centering
		\includegraphics[width=0.35\textwidth,height=0.35\textwidth]{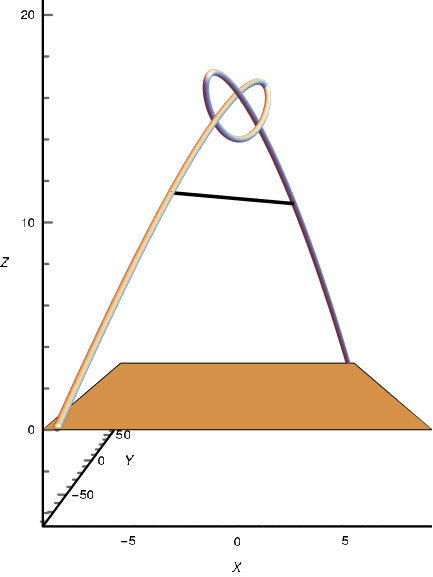}
		\caption{knotted trefoil arc with the axis of rotation}
		\label{fig: trefoil with axis}
	\end{figure}	
\noindent 
Here we define the bump function as
	\[B(t)=\frac{F(4.8-t^{2})}{F(t^{2}-3.8)+F(4.8-t^{2})}.\]
Fig. \ref{fig: Rotating about PQ} and \ref{fig:Restricted Rotation} shows how the  rotation about $PQ$ is restricted using the bump function.

		\begin{figure}[H]
				\centering
				\includegraphics[width=0.6\textwidth,height=0.2\textwidth]{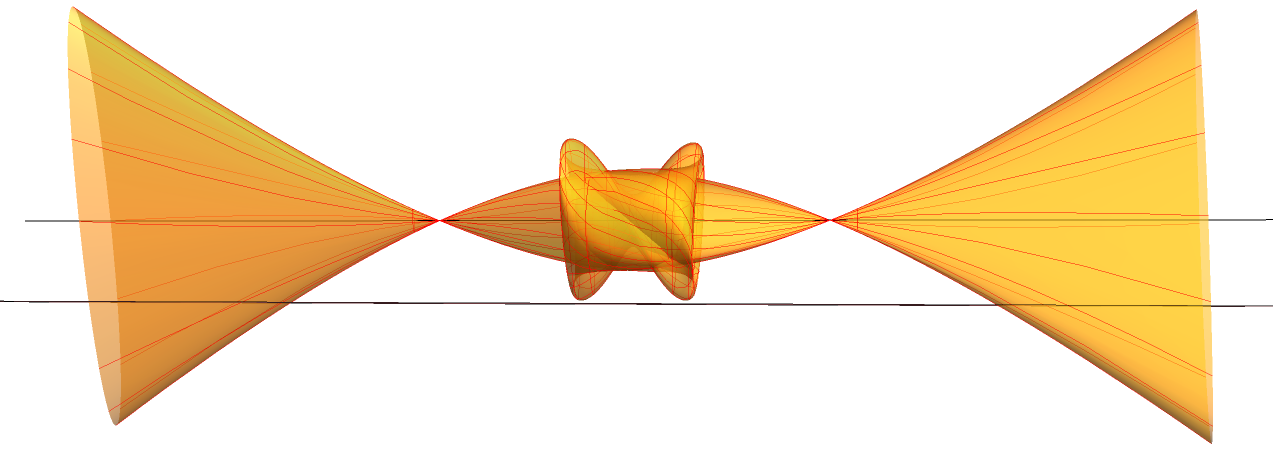}
				\caption{Rotating the long trefoil arc about $PQ$}
				\label{fig: Rotating about PQ}
			\end{figure}
	\begin{figure}[H]
		\centering
		\includegraphics[width=0.65\textwidth,height=0.3\textwidth]{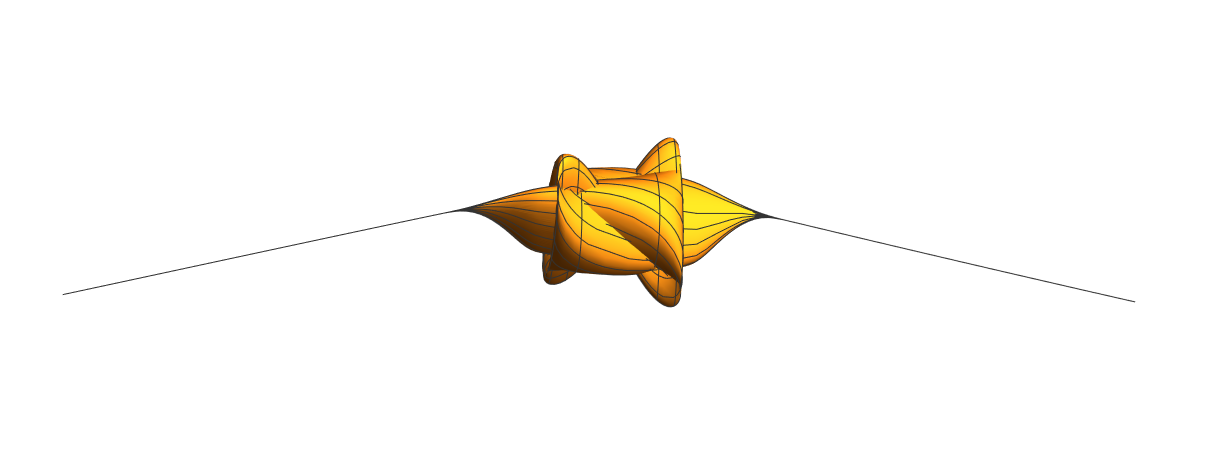}
		\caption{Restricting rotation using bump function.}
		\label{fig:Restricted Rotation}
	\end{figure}

Some of the projections of the $d$-twist spun trefoil on a hyperplane for the values $d=0,1,2,5,10,20$ are given below. In Fig. \ref{fig: zerotwist}, it is clear that $d=0$ represents Artin's spun knot. Fig. \ref{fig: zerotwistin} and Fig. \ref{fig: 1twistin} illustrate the deformation of the knotted arc while rotating around $\R^{2}$ in the twist spun knot construction, which differs from the Artin's spun knot construction. The {\it Mathematica} notebook for $d$-twist spun knot parametrization can be found in \cite{mathematica}.
\end{example}
	\begin{figure}[H]
	\centering
	\begin{subfigure}[c]{0.4\linewidth}
		\centering
		\includegraphics[width=0.8\textwidth,height=0.45\textwidth]{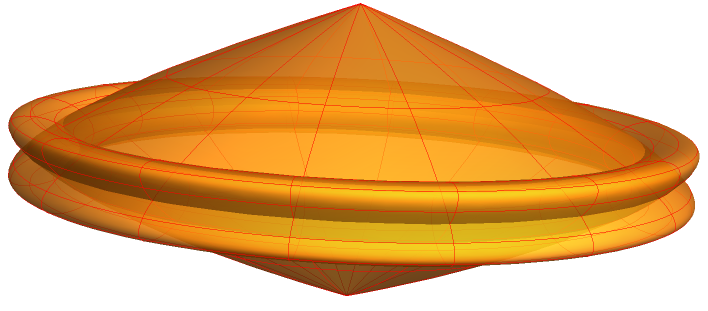}
		\caption{0-twist}
		\label{fig: zerotwist}
	\end{subfigure}
		\begin{subfigure}[c]{0.4\linewidth}
		\centering
		\includegraphics[width=0.8\textwidth,height=0.45\textwidth]{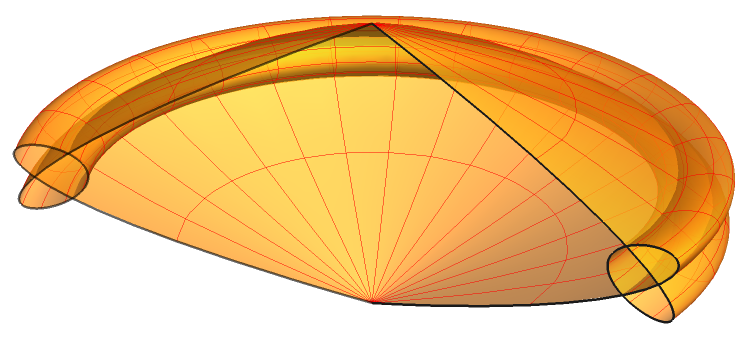}
		\caption{0-twist: inside view}
		\label{fig: zerotwistin}
	\end{subfigure}
		\caption{$d=0$}
	\end{figure}
	\begin{figure}[H]
	\centering
\begin{subfigure}[c]{0.4\linewidth}
	\centering
	\includegraphics[width=0.8\textwidth]{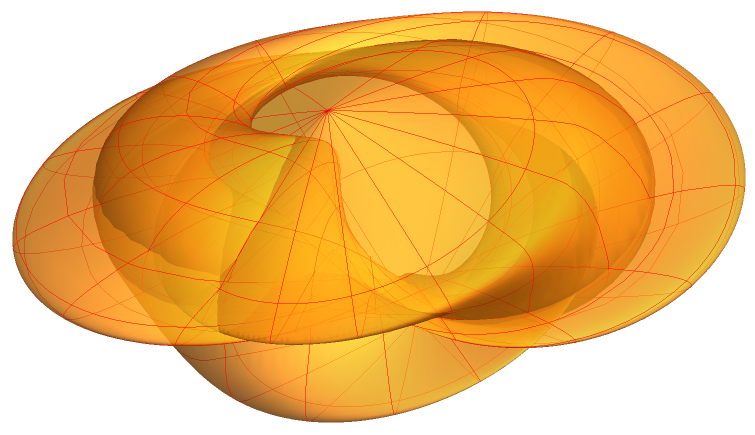}
	\caption{1-twist}
	\label{fig: 1twist}
\end{subfigure}
	\begin{subfigure}[c]{0.45\linewidth}
		\centering
		\includegraphics[width=0.8\textwidth]{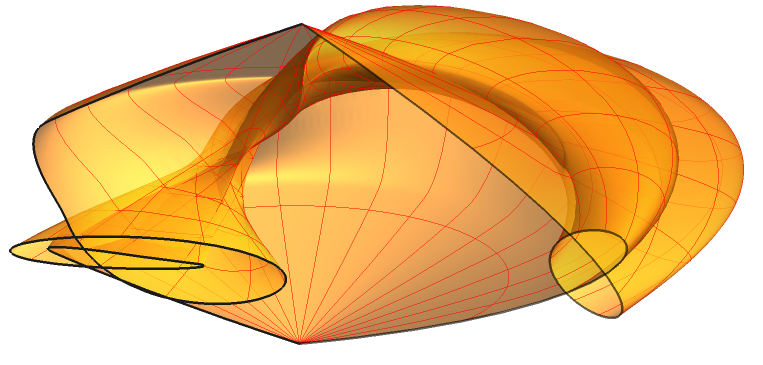}
		\caption{1-twist: inside view}
		\label{fig: 1twistin}
	\end{subfigure}
	\caption{$d=1$}
\end{figure}

	
	\begin{figure}[H]
		\centering
		\begin{subfigure}[c]{0.45\linewidth}
			\centering
			\includegraphics[width=0.8\textwidth]{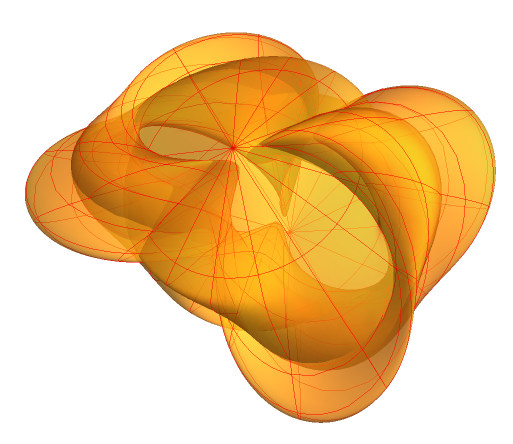}
			\caption{2-twist}
			\label{fig: 2twist}
		\end{subfigure}
		\begin{subfigure}[c]{0.45\linewidth}
			\centering
			\includegraphics[width=0.8\textwidth]{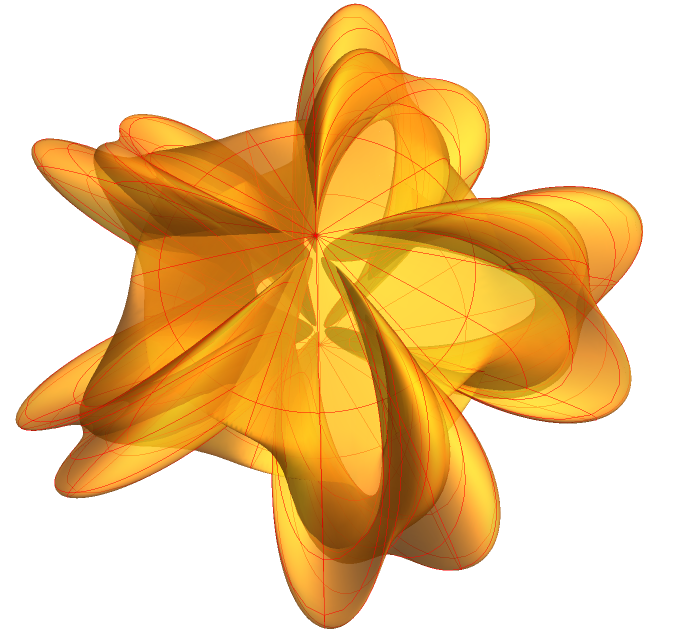}
			\caption{5-twist}
			\label{fig: 5twist}
		\end{subfigure} 
		\newline
		\begin{subfigure}[c]{0.45\linewidth}
			\centering
			\includegraphics[width=0.8\textwidth]{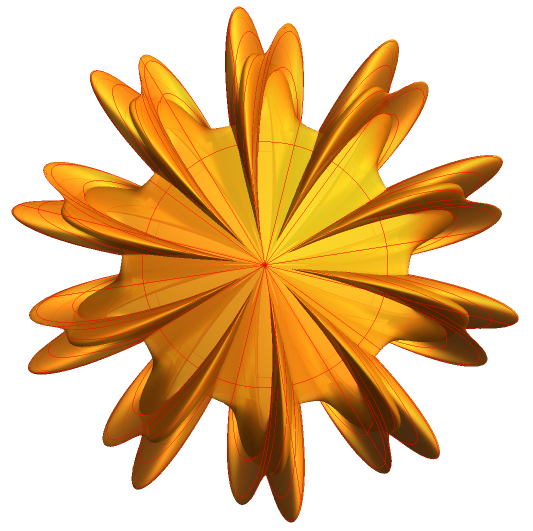}
			\caption{10-twist}
			\label{fig: 10twist}
		\end{subfigure}
		\begin{subfigure}[c]{0.45\linewidth}
			\centering
			\includegraphics[width=0.8\textwidth]{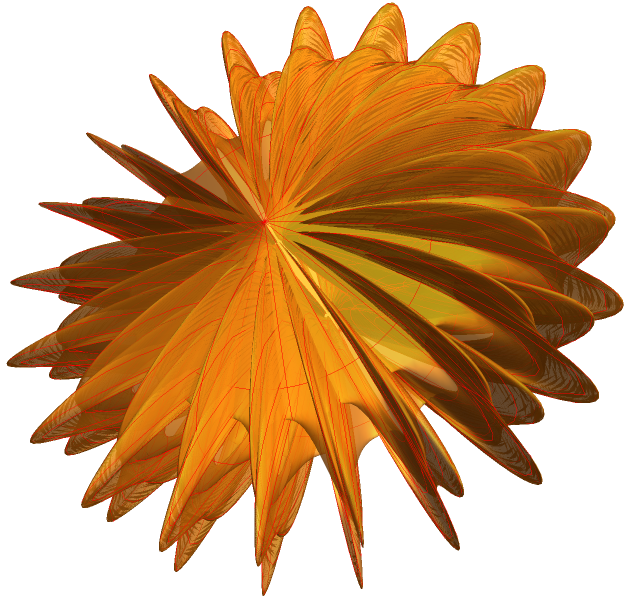}
			\caption{20-twist}
			\label{fig: 20twist}
		\end{subfigure}
		\caption{$ d=2,5,10,20.$}
	\end{figure}
	
\newpage
\section{ Parametrizing  knotted tori  } \label{sec:ribbon}
Knotted tori are ambient isotopy classes of embeddings of a torus $S^1\times S^1$ inside  $\mathbb{R}^4$.  As there can be various classes and it is very difficult to capture them in general. However, a special class of knotted tori known as {\it ribbon torus knots} (Definition \ref{def:ribbon}) are possible to parametrize due to their connection with the welded (virtual with additional moves) knots. In Section \ref{sec:weld}, we discuss briefly about welded knot theory and in Section \ref{sec:ribbonpara}, we provide a method to parametrize a ribbon torus knot.

\subsection{Welded knot theory}\label{sec:weld}

 Welded knot theory is obtained from virtual knot theory, a combinatorial generalization of classical knot theory \cite{Kau99, Sat00}. Here, a new type of crossing, called {\it virtual crossing}, is introduced in the knot diagram(Fig. \ref{fig:crossings}). A knot diagram with classical and virtual crossings is called a virtual knot diagram up to generalized Reidemeister moves. 
\begin{figure}[H]
	\centering
	\begin{subfigure}[H]{0.5\textwidth}
		\centering
		\includegraphics[width=5cm]{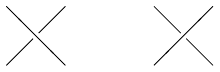}
		\caption{Classical crossings}
	\end{subfigure}
	\hspace{0.2cm}
	\begin{subfigure}[c]{0.3\textwidth}
		\centering
		\includegraphics[width=1.5cm]{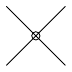}
		\caption{Virtual crossing}
	\end{subfigure}
	\caption{Two types of crossings}
	\label{fig:crossings}
\end{figure}

There are two other local moves, called {\it over} and {\it under forbidden} moves (Fig. \ref{fig:forbidden}). Allowing both on a virtual diagram results in a trivial diagram.
But allowing only the over forbidden move generates a new non-trivial knot theory, called {\it welded knot theory}. And two virtual diagrams, $K$ and $K'$, are called welded equivalent, denoted by K $\overset{w}{\sim} K'$ if they are related by a finite sequence of generalized Reidemeister moves and the over forbidden move \cite{Sat00}.  When we are in welded knot theory, we shall refer to a virtual crossing as a welded crossing. A virtual (welded) knot diagram can be thought to be obtained from a classical knot diagram in which a few crossings are converted as virtual (resp. welded) crossing.  We loosely say that these crossings have been virtualized (resp. welded). 
\begin{figure}[H]
	\centering
	\includegraphics[scale=1]{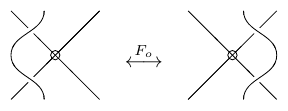}
	\qquad
	\includegraphics[scale=1]{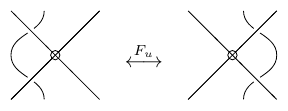}
	\caption{Forbidden moves}
	\label{fig:forbidden}
\end{figure}

S. Satoh showed a correspondence between ribbon torus knots and virtual knots by defining a map called {\it Tube} map \cite{Sat00}.
Let $K$ be a virtual knot diagram on $XY$-plane. We construct a surface diagram  corresponding to $K$ in the following way:

\begin{figure}[H]
	\centering
	\begin{subfigure}[H]{0.4\textwidth}
		\centering
		\includegraphics[width=0.8\textwidth]{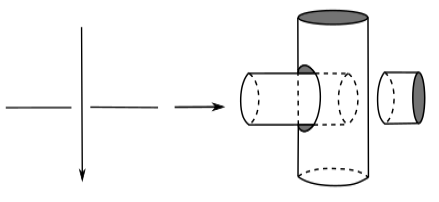}
		\caption{Near the classical crossing}
		\label{fig:tube at classical}
	\end{subfigure}
	\quad
	\begin{subfigure}[H]{0.4\textwidth}
		\centering
		\includegraphics[width=0.8\textwidth]{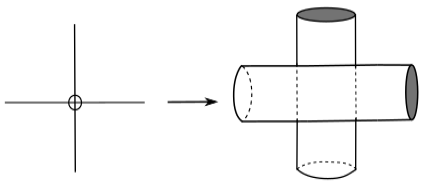}
		\caption{Near the virtual crossing}
		\label{fig:tube at virtual}
	\end{subfigure}
	\caption{Tube map near the crossings}
\end{figure}
\begin{itemize}[topsep=0pt]
	\item  At every point of the diagram, we will put a circle in the $Z$ direction.
	\item At a classical crossing, the tube corresponding to the under arc will penetrate through the tube corresponding to the over-arc (Fig. \ref{fig:tube at classical}).
	\item At the virtual crossings, the tubes are free from the over/under information of the diagram in $\R^{3}$, and they can move independently (Fig. \ref{fig:tube at virtual}).
\end{itemize}
So, $\text{Tube(K)}$ is defined as the ribbon torus knot with the surface diagram described above for $K$. He also proved the following.

\begin{proposition}\cite{Sat00}
	$ K \overset{w}{\sim} K' \implies Tube(K) \cong Tube(K').$ 
\end{proposition}

\subsection{ Parametrizing a ribbon torus knot}\label{sec:ribbonpara}

As discussed in Section \ref{sec:weld}, every ribbon torus knot is isotopic to  tube of some welded knot.  Welded knots can be created with the help of classical knots by welding few crossings. Thus from one classical knot one can obtain various non-isotopic welded knots and hence by performing tube operation, we can obtain ribbon torus knots. In Section \ref{sec:weld}, we have seen that at a classical crossing, the tube corresponding to the under-arc will penetrate through the tube corresponding to the over-arc and at a virtual crossing the tubes are free from the over/under information of the diagram in $\R^{3}$, and they can move independently. Below, we describe how to achieve this in parametrization. We prove the following theorem.

\begin{theorem}
	Given a classical knot $K$ with given trigonometric  parametrization $\phi': [0,2\pi] \rightarrow \R^3$ given by
	\[\phi'(t)=\LP f(t),g(t),h(t)\RP ,\]
	there exists functions $S_{c}(t)$ and $S_{w}(t)$ such that 
	the image of $[0,2\pi]\times [0,2\pi]$ under the map $\tilde{\Psi}: \R^2 \to \R^4$ defined by
	\[\tilde{\Psi}(t,s)=\LP f(t),\,(r-d_{1}S_{c}(t))\,g(t)+\cos s,\,(r-d_{1}S_{c}(t))\,g(t)+\sin s+d_{2}S_{w}(t),\, h(t)\RP,\]
	where $r, d_{1}, d_{2} \in \R^{+}$, is isotopic to the $Tube(K)$.
\end{theorem}

\begin{proof}
	Take the projection of $K$ on $XY$ plane given by 
	$\{f(t),g(t) \, \vert \, t \in[0,2\pi]\}$
	and let there be $n$ number of classical crossings, denoted by $c_{1},c_{2},\cdots,c_{n}$, and $m$ number of welded crossings, denoted by $w_{1},w_{2},\cdots,w_{m}$. Additionally, let $\{t_{i},s_{i}\}$  be the pair of values of $t$ where $t_{i}$ corresponds to the overcrossing point and $s_{i}$ corresponds to the undercrossing point of $c_{i}$ respectively for $i=1,\cdots, n$. Also, let $\{\bar{t}_{j},\bar{s}_{j}\}$  be the pair of values of $t$ where $\bar{t}_{j}$ corresponds the first encounter of the crossing $w_{j}$ and $\bar{s}_{j}$ corresponds the second encounter of the crossing $w_{j}$ for $j=1,\cdots, m$. \\
	First, putting a circle with a fixed radius $r \in \R_{+}$ at each point of the projection is given by
	\begin{equation*}
		\left\{\big(f(t),\,r \,g(t)+\cos s,\,r\, g(t)+\sin s \big)\, \middle| \, t,\,s \in[0,2\pi]\right\}.
	\end{equation*}
	
	Now, we construct functions for shrinking the radius of the lower tube and displacing the tubes from each other, denoted by $S_{c}(t)$ and $S_{w}(t)$ respectively.\\
	Take intervals of the same length, say $L$, around each $t_{i}$ and $s_{i}$ such that no two intervals intersect. Now, we define the bump function with centre $c$ and width $w$ (Fig. \ref{Fig:bump}) by
	\[B(x,c,w)=
	\begin{cases}
		\exp\left(-\frac{w^2}{w^2-(x-c)^2}\right) & \text{if}\; | x-c| <w \\
		0 & \text{otherwise.} 
	\end{cases}\]
		\begin{figure}[H]
		\centering 
		\includegraphics[width=0.35\textwidth]{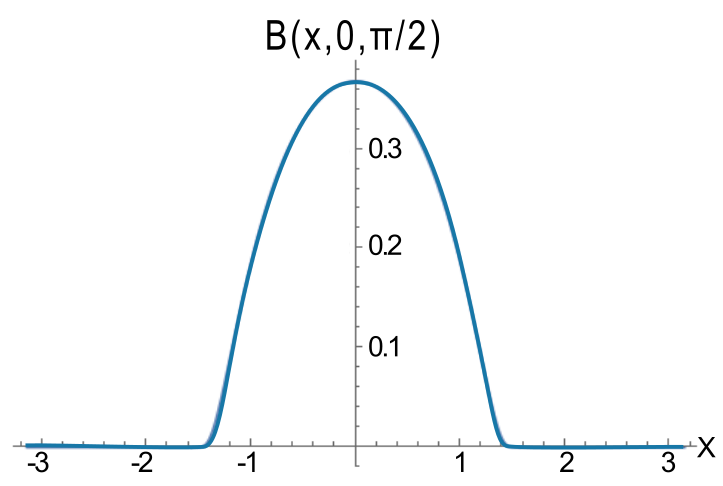}
		\caption{The bump function $B(x,c,w)$ with center $c=0$ and width $w= \frac{\pi}{2}$}
		\label{Fig:bump}
	\end{figure}
	The function $S_{c}(t)$ is defined by summing the functions $B(x,c,\frac{L}{2})$ with centres at the undercrossing points $s_{i}$ of all the classical crossings. On the other hand, the function $S_{w}(t)$ is defined by taking an alternating sum of $B(x,c,\frac{L}{2})$ with centres at the welded crossing points $\bar{t}_{j}$ and  $\bar{s}_{j}$ of all welded crossings. Thus,
		\begin{align*}
			S_{c}(x)=& \sum_{i=1}^{n} B(x,s_{i},\frac{L}{2}),\\
			S_{w}(x)=& \sum_{j=1}^{m} \LP B(x,\bar{t}_{j},\frac{L}{2})-B(x,\bar{s}_{j},\frac{L}{2})\RP.
		\end{align*}

	\noindent
	The  parametrization of the ribbon torus $Tube(K)$ is given by $\tilde{\Psi}: [0,2\pi]\times [0,2\pi] \rightarrow \R^4$ which is defined by
	\[\tilde{\Psi}(t,s)=\LP f(t),\,(r-d_{c}S_{c}(t))\,g(t)+\cos(s),\,(r-d_{c}S_{c}(t))\,g(t)+\sin(s)+d_{w}S_{w}(t),\, h(t)\RP,\]
	where $r, d_{c}, d_{w} \in \R_{+} $ are the radius of the tube, shrinking factor and the displacement factor, respectively, and $h(t)$ provides the over/under information of the tubes in $\R^{4}$. This completes the proof. 
\end{proof}

Replacing all cosine, sine and bump function with their corresponding Chebyshev approximations, we can get a polynomial  parametrization. {\it Mathematica} plots of the $Tube(K)$ near the classical and welded crossings are shown in Fig. \ref{fig:classical tube} and Fig. \ref{fig:classical tube}, respectively.
	\begin{figure}[H]
	\centering
	\begin{subfigure}[H]{0.3\linewidth}
		\centering
		\includegraphics[width=0.75\textwidth]{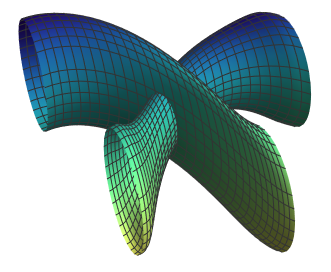}
		\caption{Near classical crossing: lower tube is shrunk}
		\label{fig:classical tube}
	\end{subfigure}
	\qquad \qquad
	\begin{subfigure}[H]{0.3\linewidth}
		\centering
		\includegraphics[width=0.65\textwidth]{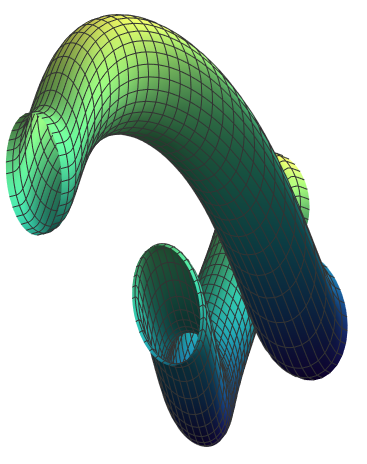}
		\caption{Near welded crossing: upper and lower tubes are displaced}
		\label{fig:welded tube}
	\end{subfigure}
	\caption{Mathematica plot of $Tube(K)$ near the crossings}
	\label{}
\end{figure}

\subsection{Example: A ribbon torus knot arising from $T(2,7)$}
In this example, we provide explicit  parametrization of the Tube of the classical torus knot $T(2,7)$ (Fig. \ref{Fig:torusknot}) with two welded crossings with one gap (Fig. \ref{Fig:weldedtorus}). Let's denote this welded knot as $K_{2,1}$.

\begin{figure}[H]
	\centering
	\begin{subfigure}[c]{0.3\linewidth}
		\centering
		\includegraphics[width=0.6\textwidth]{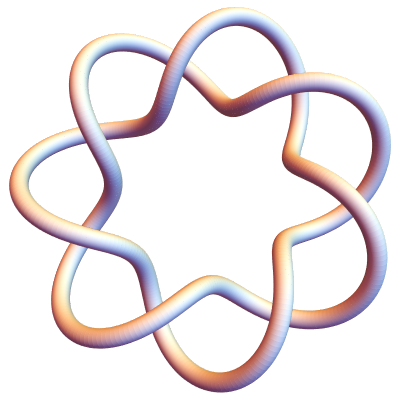}
		\caption{Classical $T(2,7)$}
		\label{Fig:torusknot}
	\end{subfigure}
	\qquad
	\begin{subfigure}[c]{0.3\linewidth}
		\centering
		\includegraphics[width=0.6\textwidth]{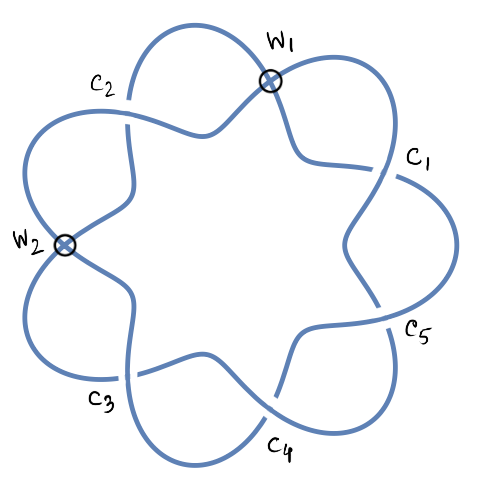}
		\caption{Welded knot $K_{2,1}$}
		\label{Fig:weldedtorus}
	\end{subfigure}
	\caption{Welded knot obtained from $T(2,7)$}
	\label{}
\end{figure}
\noindent 
The trigonometric parametrization of $T(2,7)$ is given by $	\phi'(t)=\LP f(t),g(t),h(t)\RP$, where
\begin{align*}
    f(t)&=\cos (2 t) (\cos (7 t)+3)\\
	g(t) &= \sin (2 t) (\cos (7 t)+3)\\
	h(t) &= \sin (7t).
\end{align*}

The interval pairs $[t-\frac{L}{2}, t+\frac{L}{2}]$ of equal length $L=\frac{\pi}{7}$ containing the five classical and two welded crossings with over/under information are given in table below.
\begin{table}[H]
	\centering
	\begin{TAB}(r,0.2cm,0.4cm)[3pt]{|c|c|c|c|c|c|c|c|}{|c|c|c|c|c|c|c|c|}
		\centering
		$[0,\pi]$ & $c_{1}$ & $w_{1}$ & $c_{2}$ & $w_{2}$ & $c_{3}$ & $c_{4}$ & $c_{5}$  \\
		
		over  & $\big[0,\frac{ \pi}{7}\big]$ & & $\big[\frac{2\pi}{7},\frac{3\pi}{7}\big]$ & &  $\big[\frac{4\pi}{7},\frac{5\pi}{7}\big]$ &  & $\big[\frac{6\pi}{7},\pi\big]$  \\
		
		weld & & $\big[\frac{\pi}{7},\frac{2\pi}{7}\big]$ & & $\big[\frac{3\pi}{7},\frac{4\pi}{7}\big]$ & & &  \\
		
		under &  & & & & & $\big[\frac{5\pi}{7},\frac{6\pi}{7}\big]$ & \\
		
		$[\pi, 2\pi]$ & $c_{1}$ & $w_{1}$ & $c_{2}$ & $w_{2}$ & $c_{3}$ & $c_{4}$ & $c_{5}$  \\
		
		over  &  & &  & &   &$\big[\frac{12\pi}{7},\frac{13\pi}{7}\big]$ &   \\
		
		weld & & $\big[\frac{8\pi}{7},\frac{9\pi}{7}\big]$ & & $\big[\frac{10\pi}{7},\frac{11\pi}{7}\big]$ & & &  \\
		
		under & $\big[\pi,\frac{8\pi}{7}\big]$ & &$\big[\frac{9\pi}{7},\frac{10\pi}{7}\big]$ & & $\big[\frac{11\pi}{7},\frac{12\pi}{7}\big]$&  &  $\big[\frac{13\pi}{7},2\pi \big]$
	\end{TAB}
\end{table}
 In this case, we can write 
\begin{equation*}
	\resizebox{\textwidth}{!}{$S_{c}(t)= B\LP t,\frac{11\pi}{14},\frac{\pi}{14}\RP+B\LP t,\frac{15\pi}{14},\frac{\pi}{14}\RP+B \LP t,\frac{19\pi}{14},\frac{\pi}{14}\RP+B\LP t,\frac{23\pi}{14},\frac{\pi}{14}\RP+B\LP t,\frac{27\pi}{14},\frac{\pi}{14}\RP,$}
\end{equation*}	
\begin{equation*}
	\resizebox{0.8\textwidth}{!}{$ S_{w}(t)= \LP B\LP t,\frac{3\pi}{14},\frac{\pi}{14}\RP+B\LP t,\frac{7\pi}{14},\frac{\pi}{14}\RP\RP-\LP B\LP t,\frac{17\pi}{14},\frac{\pi}{14}\RP+B\LP t,\frac{21\pi}{14},\frac{\pi}{14}\RP \RP. $}
\end{equation*}

Choosing $r=0.7, \, d_{c}=1,\, d_{w}=5 $, the parametrization of the $\text{Tube}(K_{2,1})$ is given by $\tilde{\Psi}: [0,2\pi]\times [0,2\pi] \rightarrow \R^4$ which is defined by
\begin{equation*}
\resizebox{\textwidth}{!}{$ \tilde{\Psi}(t,s)=\LP f(t),(0.7-S_{c}(t))g(t)+\cos(s),(0.7-S_{c}(t))g(t)+\sin(s)+5 S_{w}(t),h(t)\RP$}
\end{equation*}
where $ t \in [0,2\pi],s \in [0,2\pi]$. {\it Mathematica} plot of the projection of $\text{Tube}(K_{2,1})$ in $XYZ$ hyperspace is shown in Fig. \ref{Fig:tubeoftorus}.

\begin{figure}[H]
\centering 
\includegraphics[width=0.4\textwidth]{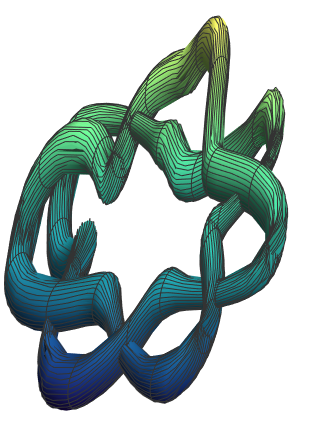}
\caption{Projection of $\text{Tube}(K_{2,1})$ in $XYZ$ hyperspace}
\label{Fig:tubeoftorus}
\end{figure}

\section{Parametrizing knotted planes}\label{sec:plane}

 In this section, we develop methods for constructing and  parametrizing knotted planes, specifically, long $2$-knots using polynomial embeddings. A \emph{long $2$-knot} refers to a proper, smooth embedding of $\R^2$ into $\R^4$, which becomes standard outside a compact region. We begin by introducing polynomial $2$-knots and prove that every such knot can be continuously deformed into a trivial surface knot via a P-homotopy. We then extend the notion of singularity index from long knots to long $2$-knots and establish that every long knot gives rise to a simple long $2$-knot. Finally, we present two geometric constructions of knotted planes via spinning, one through embedded knotted discs and the other by rotating a knotted arc with one end at infinity.
 \begin{definition}
 	A long $2$-knot is called a polynomial $2$-knot if  it is an embedding $\phi : \R^{2} \rightarrow \R^{4}$ defined by
 	\[ \phi(t,s) = (f (t,s), g(t,s), h(t,s), p(t,s)),\]
 	where $f (t,s), g(t,s), h(t,s)$ and $p(t,s)$ are real polynomials.
 \end{definition}
 \begin{theorem}\label{th:homotopy}
 	For every polynomial $2$-knot $F$, there exists a  P-homotopy that transforms $F$  to a trivial surface knot.
 \end{theorem}
 \begin{proof}
 	Let $F$ be represented by the the map $\psi: \R^{2} \rightarrow \R^{4}$ defined by 
 	\[ \phi(t,s) =\big( f(t,s), g(t,s), h(t,s), p(t,s) \big).\]
 	Up to polynomial isotopy we can assume that $p(t,s)$ to be of odd degree. Now,
 	for $u \in \R$, we define a family of maps $\phi_{u}: \R^{2} \rightarrow \R^{4}$ by
 	\[ (t,s) \rightarrow \big( f(t,s), g(t,s), h(t,s), p(t,s)+u^{2}(t+s) \big).\] 
 	Set $p(t,s)+u^{2}(t+s) = \tilde{p}_{u}(t,s)$.
 	For each $u \in \R$, the  Jacobian matrix is
 	\[\begin{pmatrix}
 		\dfrac{\partial f}{\partial t} & \dfrac{\partial f}{\partial s} \\[0.3cm]
 		\dfrac{\partial g}{\partial t}  & \dfrac{\partial g}{\partial s} \\[0.3cm]
 		\dfrac{\partial h}{\partial t} & \dfrac{\partial h}{\partial s} \\[0.3cm]
 		\dfrac{\partial p}{\partial t}+u^{2} & \dfrac{\partial p}{\partial s}+ u^{2}
 	\end{pmatrix}\]
 	has rank $2$ since $\psi$ is an immersion.
 	Since $p(t,s)$ is of odd degree, all partial derivatives of $p(t,s)$, $\frac{\partial p}{\partial t}$ and $\frac{\partial p}{\partial s}$ are even degree polynomials. As shown in  \cite{Mis14}, there exist $R_{1},R_{2} \in \R$ such that
 	$\frac{\partial p}{\partial t} +R_{1} >0$ and $\frac{\partial p}{\partial s} +R_{2} >0$\\
 	Take $R =\max \{R_{1},R_{2}\}$. Therefore
 	
 	\[	\frac{\partial p}{\partial t} +R  >0\quad \hfill
 		\text{and} \quad \frac{\partial p}{\partial s} +R  >0 .\]
 	
 	Let $M= \sqrt{R}$. Therefore,	for $u \geq M$,
 	\begin{align*}
 		\dfrac{\partial \tilde{p}_{u}(t,s)}{\partial t}&= \dfrac{\partial p_{u}(t,s)}{\partial t}+ u^{2}  > 0,\\
 		\text{and} \quad \dfrac{\partial \tilde{p}_{u}(t,s)}{\partial s}&= \dfrac{\partial p_{u}(t,s)}{\partial s}+ u^{2}  > 0 .
 	\end{align*}   
 	i.e., $\tilde{p}_{u}(t,s)$is monotonically increasing for $u \geq M$, ensuring that $\phi_u$ is a trivial surface knot.\\
 	Therefore, we have a homotopy $\phi: \R^{2} \times [0,M] \rightarrow \R^{4}$ defined by 
 	\[\phi(t,s,u)=\big( f(t,s), g(t,s), h(t,s), p(t,s)+u^{2}(t+s) \big)\]
 	such that $\phi(t,s,0)=\big( f(t,s), g(t,s), h(t,s), p(t,s)\big)$ is the original long 2-knot $F$ and $\phi(t,s,M)=\big( f(t,s), g(t,s), h(t,s), p(t,s)+ M^{2}(t+s)\big)$ presents a trivial long 2-knot. And for each $u \in [0,M]$, $\phi(t,s,u)$ is an immersion. 
  \end{proof}
 To understand how such homotopies pass through singular stages, we define the singularity index, generalizing the analogous notion from classical long knots \cite{Mis14}.
 
 \begin{remark}
 	In Theorem \ref{th:homotopy}, for each $u \in [0,M]$ we get an immersion $\phi_{u}$ in $\R^{4}$ with same projection as $F$ in $\R^{3}$. This implies that the singular set is same for each $\phi_{u}$ and the one parameter family $\{p(t,s)+u^{2}(t+s) \, \vert \, u \in [0,M]\}$ changes the over/under information through the homotopy $\phi$. Therefore, there are finitely many $ u \in [0,M]$ for which $\phi_{u}$ is an immersion in $\R^{4}$. Let us refer them as  singular long $2$-knots. The number of these singular long $2$-knots in the homotopy $\phi$ is called the singularity index of the polynomial representation $\psi$ of $F$.
 \end{remark}
 \begin{definition}
 	{\it Singularity Index} of a  long $2$-knot $F$ is defined to be the minimum number of singularity indices over all polynomials $\psi$ presenting $F$ and denoted by {\it SI(F)}.
 \end{definition}
 \begin{definition}
 	A long $2$-knot is said to be {\it simple} if it has a projection on some $\R^3$ such that a singular set consists of finitely many double point curves only. 
 \end{definition}
 Next, we establish that any classical polynomial knot can be used to construct a simple long $2$-knot whose singularity index is less than or equal to that of the original knot.
 \begin{theorem}\label{th:sing_index}
 	Given a polynomial knot $K$, with the  singularity index $n$ there exists a simple polynomial $2$-knot $P_K$ with the singularity index less than or equal to $n$.
 \end{theorem}
 \begin{proof}
 	Let $K : \R \rightarrow \R^{3}$ be a polynomial knot defined as $K(t) = (f (t), g(t), h(t))$ with singularity index $n$, where $(f(t),g(t))$ is an immersion in $\R^{2}$ and $h(t)$ is an odd degree polynomial with positive co-efficient [Remark 2.4, \cite{Mis14}].
 	Now, we construct a long 2-knot given by $P_{K}: \R \times \R \rightarrow \R^{4}$ and defined by
 	\[P_{K}(t,s)= (f(t)+s, g(t)+s, s, h(t)).\]
 	$K$ being an immersion, the  Jacobian matrix for $P_{K}$ 
 	\[\begin{pmatrix}
 		\dfrac{\partial f}{\partial t} & 1 \\[0.3cm]
 		\dfrac{\partial g}{\partial t} & 1 \\[0.2cm]
 		0 & 1 \\[0.1cm]
 		\dfrac{\partial h}{\partial t} & 1 
 	\end{pmatrix}\]
 	has rank two. Therefore, $P_{K}$ is an immersion.
 	
 	Suppose $(t_{1},s_{1})$ and $(t_{2},s_{2})$ be such that
 		\begin{align*}
 		f(t_{1})+s_{1}&=f(t_{2})+s_{2},\\
 		g(t_{1})+s_{1}&=g(t_{2})+s_{2}, \\
 		s_{1}&= s_{2}, \\
 		h(t_{1}) &= h(t_{2}).
 	\end{align*}
 	This implies $
 		f(t_{1})=f(t_{2}),\;
 		g(t_{1})=g(t_{2}),\;
 		h(t_{1}) = h(t_{2}).$
 	Since, $K$ is an embedding, we have $ t_{1}=t_{2}$.
 	Hence, $(t_{1},s_{1}) = (t_{2},s_{2})$ and therefore $P_{K}$ is an injective immersion  i.e an embedding.
 	The singular set of the long $2$-knot contains only double curves, one for each double point of $K$. If $(t_{j}, t_{k})$ are the values of $t$ for a crossing in the projection of $K$ then the pair $(t_{j},s)$ and $(t_{k},s)$ will provide a double curve in the projection of $P_{K}$ in $\R^{3}$.
 	
 	By Theorem \ref{th:homotopy}, we have a continuous map $\Psi : \R^{2} \times [0, R] \rightarrow
 	\R^{4}$ such that $\Psi(t,s, 0) = (f (t)+s, g(t)+s,s, h(t))$ is the given knot $P_K$ and $\Psi(t,s, R) = (f (t) +s, g(t) +s,  s, h(t) + R^{2} (t+s))$ is a trivial knot and for each $u \in [0, R]$ the map $\Psi((t, s),u) = 
 	(f (t) + s, g(t) + s, s, h(t) + u^{2} (t+s))$ is an immersion.
 	Now, restriction of this homotopy $\Psi$ on $\R \times [0, R]$, i.e., $\Psi \vert_{\R \times [0,R]} \R \times [0,R] \rightarrow \R^{3}$ given by
 	\[ (t,u) \rightarrow (f(t)+u,g(t)+u, h(t)+ u^{2}t),\]
 	is the homotopy that transforms the long knot $K$ to a trivial classical knot. Now, $K$ has singularity index $n$. Therefore, there are at least $n$ values of $u$, say $u_{1}, \cdots u_{n}$ for which $\Psi(t,u_{i})$ is a singular knot. For those values of $u_{i}, \, i=1, \cdots, n$, 
 	\begin{align*}
 		f(t_{a})+u_{i}&=f(t_{b}) +u_{i},\\
 		g(t_{a})+u_{i}&=g(t_{b}) +u_{i},\\
 		h(t_{a}) +u_{i}^{2} t_{a}&= h(t_{b}) +u_{i}^{2} t_{b},	
 	\end{align*} for all pairs $(t_{a},t_{b})$ representing the crossing points. This implies 
 	\begin{align*}
 		f(t_{a})+u_{i}+s&=f(t_{b}) +u_{i}+s,\\
 		g(t_{a})+u_{i}+s&=g(t_{b}) +u_{i}+s,\\
 		h(t_{a}) +u_{i}^{2} t_{a}&= h(t_{b}) +u_{i}^{2} t_{b},	
 	\end{align*} for all the pairs $(t_{a},s)$, $(t_{b},s)$, representing the double curves.\\
 	This suggests that $\Psi(t,s,u_{i})$ corresponding to $\Psi(t,u_{i})$ for $u_{i},\, i=1, \cdots, n$ are not embeddings in $\R^{4}$. This indicates the fact that the long $2$-knot $P_{K}$ is transformed to the trivial surface knot diagram by exchanging over/under information in the set of $n$ number of double curves. Hence, the singularity index of $P_{K}$ is less or equal to $n$.
 \end{proof}
 \begin{example}	Here is an example of a long $2$-knot (Fig. \ref{fig:simplelong 2 knot}) constructed from the long trefoil given by the map \[ (t,s)\to (t^3-3t+s,t^4-4t^2+s,s,t^5-10t).\]
 \end{example}
 \begin{figure}[H]
 	\centering
 	\includegraphics[width=0.35\textwidth]{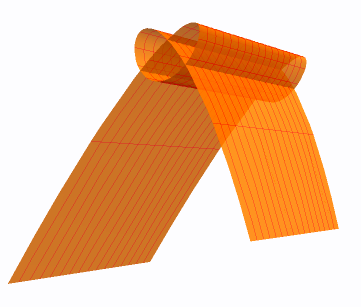}
 	\caption{Projection of the simple long trefoil 2-knot on a hyperplane}
 	\label{fig:simplelong 2 knot}
 \end{figure}
 We now present two constructions of knotted planes in $\R^4$, both derived from a given long knot $K$ using spinning operations.

 \subsection{Constructing knotted planes by spinning }
 
 {\bf Construction I:}
  This construction starts with a knotted disc bounded by the connected sum $K \sharp K^*$ and extends it to a full knotted plane via a polynomial homotopy.
 \begin{lemma}\label{lem:disc}
 	For every polynomial $K$, there exists a knotted disc bounded by $K \sharp K^{*}$.
 \end{lemma}
 \begin{proof}
 	First, take a knotted arc $k$ of the polynomial knot $K$ in the upper half space $\R^{3}_{+}$ with only two endpoints on the boundary $\R^{2}$. Let the knotted arc be presented by $ \big(f(t),g(t),h(t)\big)$ where $t \in [a,b]$ where $h(a)=0=h(b)$ and $h(t) >0 \quad \text{for} \, t\in (a,b)$. Spinning this arc around $\R^{2}$ from angle $0$ to $2\pi$, results into a semi sphere or a $2$-disc in $\R^{4}$ bounded by $K \sharp K^{*}$ where $K^{*}$ is the mirror image of $K$. Denote it by $D$. Then $D$ can be  parametrized by 
 	\[[a,b] \times [0,\pi] \rightarrow \R^4 \]
 	\[(t,s) \rightarrow \LP f(t),g(t),h(t) \sin s, h(t) \cos s \RP,\]
 	where $s=0$ and $s=\pi$ will map to the knotted arcs $k$ and $k^{*}$ on the boundary of $D$, given by
 	\[ \LP f(t),g(t),0, h(t)\RP\, \text{for} \, t\in [a,b],  \]
 	\[ \LP f(t),g(t),0, -h(t)\RP\, \text{for} \, t\in [a,b] \]
 	respectively.
 \end{proof}
 \begin{example}
 	Here is an example of a knotted disc with boundary $K \sharp K^{*}$ where $K$ is the long trefoil (Fig. \ref{fig:disc}), given by 
 	\[[a,b] \times [0,\pi] \rightarrow \R^4 \]
 	\[(t,s) \rightarrow \LP t^3-3t,t^5-10t,(-t^4 + 4 t^2 + 3) \sin s,(-t^4 + 4 t^2) + 3) \cos s \RP\]
 \end{example}
 \begin{figure}[H]
 	\centering
 	\includegraphics[width=0.35\textwidth]{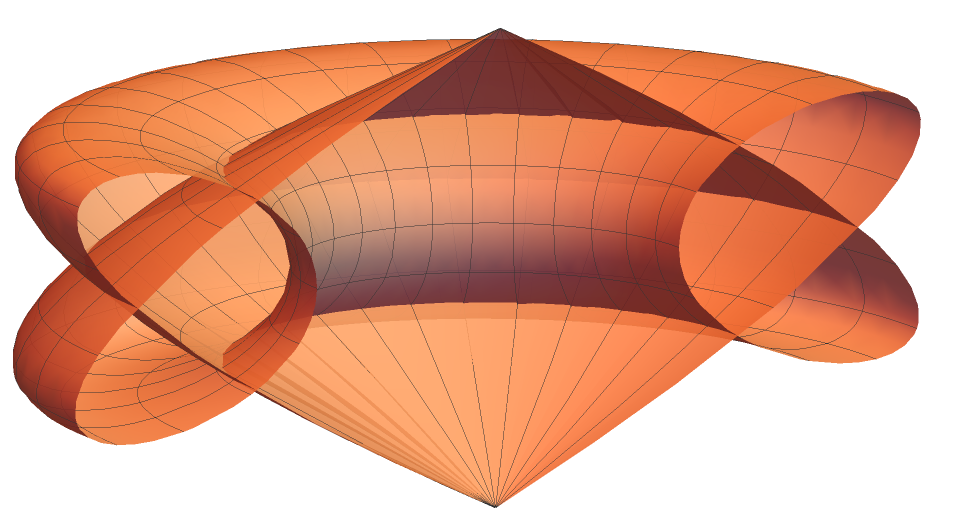}
 	\caption{Knotted disc with boundary $K \sharp K^{*}$}
 	\label{fig:disc}
 \end{figure}
 \begin{theorem}
 	For every polynomial knot $K$, there exists a knotted plane that contains an knotted disc bounded by $K \sharp K^{*}$.
 \end{theorem}
 \begin{proof}
 	By Lemma~\ref{lem:disc}, we have a knotted disc $D$ with boundary $K \sharp K^{*}$, given by 
 	$\phi: [a,b] \times [0,2\pi] \rightarrow \R^{4},$ where
 	\[\phi(t,s) = \big(f(t),g(t),h(t) \sin s, h(t) \cos s\big).\]
 	Now we construct an annulus $A$ with boundary $K \sharp K^{*}$ in one end and an unknotted $S^{1}$ in the other end by using the P-homotopy map mentioned in Theorem \ref{th:sing_index}.
 	For $ t\in[a,b],\, s\in [-\infty,0]$, we define
 		\begin{align*}
 		P_{K}:& \;\psi_{1}(t,s) = \big(f(t)-s,g(t)-s, s, h(t)+s^{2}t \big) ,\\
 		P_{K^{*}}:& \; \psi_{2}(t,s)= \big(f(t)+(s-\pi),g(t)+(s-\pi),-(s-\pi), -h(t)+(s-\pi)^{2}t \big),
 	\end{align*}
 	where $\psi_{1}(t,s)$ and $\psi_{2}(t,s)$ are trivial diagrams for all $s \leq -R$ and $s \geq \pi+R$ by [Lemma 2.7, \cite{Mis14}]. Also, we have the following:
 	\begin{itemize}[itemsep=3pt]
 		\item $\psi_{1}(t,0)= \phi(t,0)$, $\psi_{2}(t,\pi)=\phi(t,\pi)$ presents $k$ and $k^{*}$ respectively.
 		\item $k$ and $k^{*}$ are joined at the endpoints $\psi_{1}(a,0)= \psi_{1}(a,0)$, $\psi_{2}(b,\pi)=\psi_{2}(t,\pi)$.
 		\item The values of $\psi_{1}(a,s)$ for $s\in [-\infty,0]$ and $\psi_{2}(b,s)$ for $s \in [\pi,\infty]$ are equal since $h(a)=0=h(b)$. Therefore, $P_{K}$ and $P_{K^{*}}$ are joined along the lines $\psi_{1}(a,s),\, s\in [-\infty,0]$ and $\psi_{2}(b,s),\, s \in [\pi,\infty]$. \item Then the annulus $A$ is obtained by restricting $P_{K}\, \cup \, P_{K^{*}}$ to $s \in [-R,0] \cup [0,R]$. 
 		\item The other end of the annulus $A$ is a unknotted circle created by $\psi_{1}(t,-R)$ and $ \psi_{2}(t,\pi+R) $ which are unknotted arcs joined at the endpoints $\psi_{1}(a,-R)=\psi_{1}(a,\pi+R)$ and $\psi_{2}(a,-R)=\psi_{2}(a,\pi+R)$.
 		\item Outside the interval $s \in [-R,R]$, $\psi_{1}(t,s) \cup \psi_{2}(t,s) $ will always be a unknotted $S^{1}$.
 	\end{itemize}
 	So, $D \cup A$ joined along the boundary $K \sharp K^{*}$   gives us an embedding of an open knotted disc with boundary an unknotted $S^{1}$. Therefore, extending the interval for $s$ from $[-R,R]$ to $\R$ we get a  parametrization of a knotted plane given by 
 	\[\Phi: (a,b) \times 
 	\R \rightarrow \R^{4}\] 
 	\[\Phi=\begin{cases}
 		\psi_{1} & s \in (-\infty,0] \\
 		\phi & s\in [0, \pi] \\
 		\psi_{2} & s \in [\pi, \infty)
 	\end{cases}  \]
 	This completes the proof.
 \end{proof}
 
 \begin{remark}
 	This construction shows that $K \sharp K^*$ bounds a smooth disc in $D^4$, hence it is a slice knot. \cite{CarSai}.
 \end{remark}
 \begin{example}
 	For the long trefoil knot $K$, we have the knotted plane  given by 
 	$\Phi: (a,b) \times 
 	\R \rightarrow \R^{4},$ which is defined by 
 	\[\Phi=\begin{cases}
 		\psi_{1} & s \in (-\infty,0], \\
 		\phi & s\in [0, \pi], \\
 		\psi_{2} & s \in [\pi, \infty),
 	\end{cases}  \]
 	\begin{align*}
 		\text{where} \quad \phi(t,s)&=\big(t^3-3t,t^5-10t,(-t^3+4t^2+3) \sin s, (-t^3+4t^2+3) \cos s\big),\\
 		\psi_{1}(t,s) &= \big(t^3-3t-s,t^5-10t-s, s, (-t^3+4t^2+3)+s^{2}t \big),\\
 		\psi_{2}(t,s)&= \big(t^3-3t+(s-\pi),t^5-10t+(s-\pi),-(s-\pi), (t^3-4t^2-3)+(s-\pi)^{2}t \big).
 	\end{align*}
 	In Fig. \ref{fig: slice plane}, $D$ is shown in orange color and the annulus $A$ is shown in blue and green for $P_{k}$ and $P_{K^{*}}$ respectively.
 \end{example}
 \noindent
 \textbf{Construction II:}
 This method spins an arc of a long knot $K$ extending its one end to the infinity in $\R^3_+$ to produce a knotted plane in $\R^4$.
  \begin{figure}[H]
 	\centering
 	\includegraphics[width=0.5\textwidth]{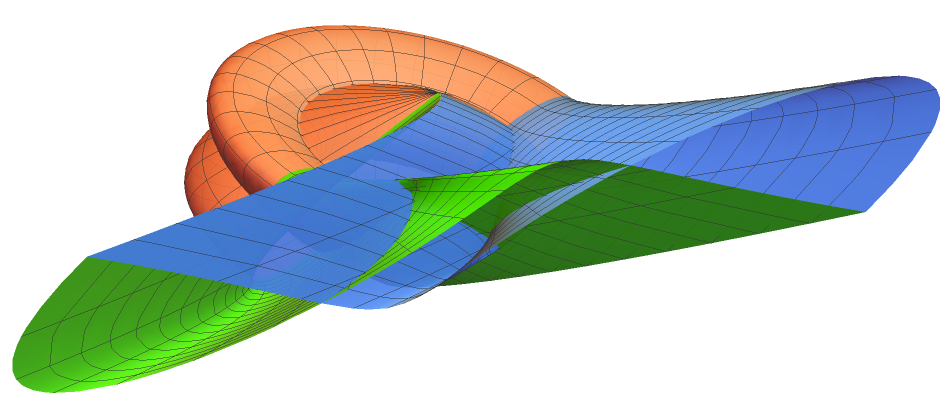}
 	\caption{Knotted plane with a knotted disc bounded by $K \sharp K^{*}$.}
 	\label{fig: slice plane}
 \end{figure}
 \begin{theorem}
 	For a polynomial knot $K$ given by
 	$\big\{(f(t),g(t),h(t))\,\vert\, t \in \R\big\}$,
 	there exists a polynomial 2-knot obtained from spinning operation with  parametrization given by $\psi': [a,\infty] \times [0,2\pi] \rightarrow \R^4$ which is defined by 
 	\[\psi'(t,s)=\LP f_{1}(t),g_{1}(t),h_{1}(t)\,S(s), h_{1}(t)\,C(s)\RP.\]
 \end{theorem}
 \begin{proof}
 	We start with the knotted arc of $K$ in the upper half space $\R^{3}_{+}$. Instead of taking both endpoints of the knotted arc on $\R^{2}$, we put one endpoint on $\R^{2}$, and the other end goes to infinity. This happens when $h(t)$ is an odd degree polynomial. If $h(t)$ is an even degree, then we can add a term of a larger odd degree with a very small coefficient so that the over/under information is not disturbed (Fig. \ref{fig:knotinf}).
 	Let this knotted arc be given by
 	\[[a,\infty) \rightarrow \R^3\]
 	\[t \rightarrow \LP f_{1}(t),g_{1}(t),h_{1}(t)\RP .\]
 	
 	\begin{figure}[H]
 		\centering 
 		\includegraphics[width=0.2\textwidth,height=0.2\textwidth]{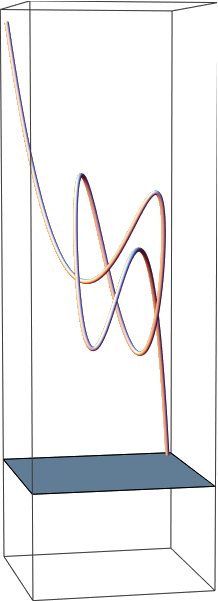}
 		\caption{Long trefoil arc for $t\in [a,\infty]$}
 		\label{fig:knotinf}
 	\end{figure}
 	Now, rotating this knotted arc around $\R^{2}$ results in a knotted plane in $S^{4}$ which can be  parametrized by 
 	\[[a,\infty) \times [0,2\pi] \rightarrow \R^4\]
 	\[t \rightarrow \big(f_{1}(t),g_{1}(t),h_{1}(t) \sin s, h_{1}(t) \cos s\big).\]
 	Replacing $\cos s$ and $\sin s$ with their corresponding Chebyshev approximations $C(s)$ and $S(s)$, respectively, we get the final polynomial  parametrization. This completes the proof.
 \end{proof}
 \begin{example}
 	An example of a knotted plane obtained from a long trefoil is shown in Fig. \ref{fig: spin plane}.
 	
 	\begin{figure}[H]
 		\centering
 		\begin{subfigure}[H]{0.4\linewidth}
 			\centering
 			\includegraphics[width=\textwidth,height=0.6\textwidth]{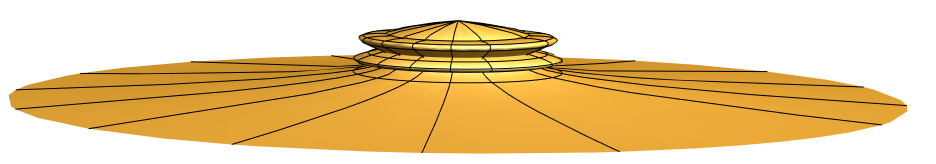}
 			\caption{ The knotted part inside a compact region }
 		\end{subfigure}
 		\quad
 		\begin{subfigure}[H]{0.4\linewidth}
 			\centering
 			\includegraphics[width=\textwidth,height=0.5\textwidth]{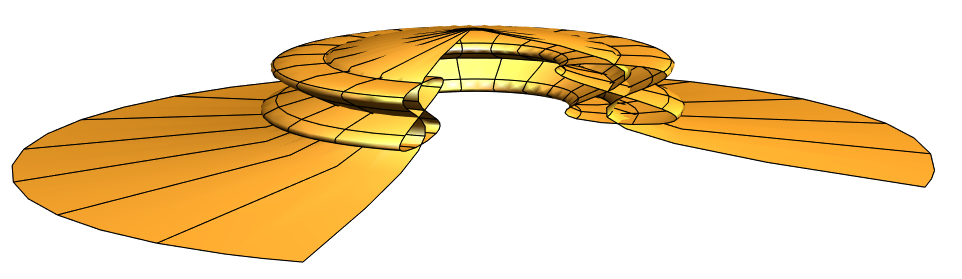}
 			\caption{Inside view}
 		\end{subfigure}
 		\caption{Knotted plane obtained by spinning} 
 		\label{fig: spin plane}
 	\end{figure}
 \end{example}

 \section{Conclusion:}
 Our constructions provide infinitely many example of knotted surfaces. In case of knotted spheres and knotted tori these surface knots are compact. However,
 Knotted planes are non-compact. We have seen that their constructions rely on classical knots, once the parametrization of classical knots is known these knotted surfaces are easily parametrized and their geometric properties can be studied. One future goal is to extend parametrization to broader classes of knotted surfaces. This includes non‐orientable embeddings (like projective planes or Klein bottles in four-dimensional space), surfaces of higher genus, or surfaces with branch points. Moreover, a degree‐minimization problem i.e finding the lowest polynomial degree that realizes a given knotted surface is another problem to study in future.
 
 \section{Acknowledgment}  
 The second author gratefully acknowledges financial support from the Council of Scientific and Industrial Research (CSIR), India.

\end{document}